%% file: main.tex
\documentclass[11pt]{article}
\pdfoutput=1

\usepackage[bib,todo]{preamble/packages}
\usepackage{preamble/commands}
\usepackage{preamble/formatting}

\usepackage{csquotes}
\usepackage{titlesec}
\usepackage{subfiles}

\title{The K-theory cochains of H-spaces and height 1 chromatic homotopy theory}
\author{Sven van Nigtevecht\footnote{\href{mailto:s.vannigtevecht@uu.nl}{\texttt{s.vannigtevecht@uu.nl}}}, with a joint appendix with Max Blans\footnote{\href{mailto:m.a.blans@uu.nl}{\texttt{m.a.blans@uu.nl}}}}
\date{4th March 2025}

\newcommand*{\K}{\mathrm{K}}
\newcommand*{\E}{\mathbf{E}}
\newcommand*{\Q}{\mathbf{Q}}
\newcommand*{\Z}{\mathbf{Z}}
\newcommand*{\N}{\mathbf{N}}

\newcommand*{\nerve}{\mathrm{N}}

\titleformat{\section}
  {\normalfont\Large\bfseries\boldmath}{\thesection}{1em}{}

\titleformat{\subsection}
  {\normalfont\large\bfseries\boldmath}{\thesubsection}{1em}{}

\newlist{condenum}{enumerate}{1}
\setlist[condenum]{label={\upshape(}\roman*{\upshape)}}

\newlist{thmenum}{enumerate}{1}
\setlist[thmenum]{label={{\upshape(}\alph*{\upshape)}}}

\newlist{defenum}{enumerate}{1}
\setlist[defenum]{label={{\upshape(}\alph*{\upshape)}}}

\newcommand*{\Modh}{\smash{\widehat{\mathrm{Mod}}}\vphantom{\Mod}}
\newcommand*{\LModh}{\smash{\widehat{\mathrm{LMod}}}\vphantom{\LMod}}
\newcommand*{\RModh}{\smash{\widehat{\mathrm{RMod}}}\vphantom{\RMod}}
\newcommand*{\CAlgh}{\smash{\widehat{\mathrm{CAlg}}}\vphantom{{\mathbf{CAlg}}}}
\newcommand*{\Torh}{\smash{\widehat{\Tor}}\vphantom{\Tor}}

\DeclareMathOperator{\LTor}{LTor}
\DeclareMathOperator{\RTor}{RTor}
\newcommand*{\LTorh}{\smash{\widehat{\LTor}}\vphantom{\LTor}}
\newcommand*{\RTorh}{\smash{\widehat{\RTor}}\vphantom{\RTor}}

\DeclareMathOperator{\Free}{Free}

\newcommand*{\compKU}{\mathop{(\KU_p)^{\wedge}_*}}

\newcommand*{\MorMod}{\mathrm{MorMod}}

\newcommand*{\hLambda}{\hat{\Lambda}}
\newcommand*{\tM}{\widetilde{M}}

\begin{document}
\maketitle

\begin{abstract}
\noindent
Fix an odd prime $p$.
Let $X$ be a pointed space whose $p$-completed K-theory $\mathrm{KU}_p^*(X)$ is an exterior algebra on a finite number of odd generators; examples include odd spheres and many H-spaces.
We give a generators-and-relations description of the $\mathbf{E}_\infty$-$\mathrm{KU}_p$-algebra spectrum $\mathrm{KU}_p^{X_+}$ of $\mathrm{KU}_p$-cochains of $X$.
To facilitate this construction, we describe a $\K(1)$-local analogue of the Tor spectral sequence for $\mathbf{E}_1$-ring spectra.
Combined with previous work of Bousfield, this description of the cochains of $X$ recovers a result of Kjaer that the $v_1$-periodic homotopy type of $X$ can be modelled by these cochains.
This then implies that the Goodwillie tower of the height~$1$ Bousfield--Kuhn functor converges for such $X$.
\end{abstract}

\setcounter{tocdepth}{1}
\tableofcontents

\section{Introduction}
If $A$ is an $\E_\infty$-ring spectrum and $X$ a space, then there is an $\E_\infty$-ring spectrum $A^{X_+}$ (called the \emph{$A$-cochains} of $X$) refining the $A$-cohomology of $X$:
\[
    \mathop{\pi_{-*}} A^{X_+} \cong A^*(X).
\]
If $A = \uH R$ is the Eilenberg–MacLane spectrum of a commutative ring $R$, then this recovers the $R$-valued singular cochains $C^\bullet(X; R)$ if we use the equivalence between the $\infty$-categories of $\uH R$-modules and of chain complexes over $R$.

The ring spectrum $A^{X_+}$ captures more information about $X$ than does the graded ring $A^*(X)$.
This can already be clearly seen in the case of rational cohomology: from the works of Quillen \cite{quillenRationalHomotopyTheory1969} and Sullivan \cite{sullivanInfinitesimalComputationsTopology1977} we can deduce that the rational cochains give rise to an equivalence of $\infty$-categories
\[
    \Spaces_\Q^{\geq 2,\ \fin} \simeqto \CAlg^\aug(\Ch_\Q)^{\geq 2,\ \fin}
\]
between simply connected rational pointed spaces of finite type, and augmented simply connected rational cdga's of finite type.

In the present article, we consider the case where $A$ is $p$-completed complex K-theory $\KU_p$, where $p$ is an odd prime.
Our first result (Theorem~\ref{thm:cochain_presentation}) is a computation of the $\KU_p$-cochains of a class of H-spaces, by giving a presentation in terms of `generators-and-relations'.
This relies on the fact that the $p$-adic K-theory of many H-spaces is an exterior algebra on a finite number of odd generators.
In fact, this is all we use of such H-spaces, and so our computation works more generally for any space with such $p$-adic K-theory.

Similar to the case of the rational cochains, the K-theory cochains of a (finite) space $X$ also capture a part of the homotopy type of $X$: the \emph{$v_1$-periodic} part.
This is a special case of work by Behrens and Rezk \cite{Behrens_Rezk_Kn} and Heuts \cite{Heuts_Lie} describing higher algebraic models for unstable homotopy theory.
Briefly, in unstable chromatic homotopy theory, the $\infty$-category of $p$-local pointed spaces $\Spaces_p$ is studied through a sequence of localisations $\Spaces_{v_n}$, one for each $n\geq 0$.
Here $n$ is referred to as the \emph{height}.
Height~0 corresponds to rationalisation.
Behrens and Rezk describe a cochain-model for $v_n$-periodic spaces, whereas Heuts works with coalgebras and spectral Lie algebras.
For an introduction to the subject, we refer to \cite{Behrens_Rezk_survey_Kn} and \cite{Heuts_handbook_Lie}.

The main difference compared to rational homotopy theory is that the cochain model does not yield an equivalence of $\infty$-categories if $n>0$: appropriate cochains (e.g., for $n=1$, the $\KU_p$-cochains) only provide an approximation to $v_n$-periodic homotopy theory.
The failure of this approximation can be measured: roughly speaking, the cochains work well if the value of the (height~$n$) \emph{Bousfield--Kuhn functor} $\Phi_n$ at $X$ can be recovered from the cochains.
More precisely, the topological Andr\'e--Quillen cohomology (henceforth \emph{TAQ-cohomology}) of the cochains on $X$ should agree with $L_{\K(n)}\Phi_nX$ via a specific \emph{comparison map}.
For finite $X$ this turns out to be equivalent to asking the Goodwillie tower of $\Phi_n$ to converge at $X$.
Spaces for which this tower converges are called \emph{$\Phi_n$-good} in this context; see \cite[§8]{Behrens_Rezk_survey_Kn} for a further discussion of $\Phi_n$-good spaces.
See \cite{kuhn_guide_telescopic} for background on the functors $\Phi_n$.
 
It was previously shown by Kjaer \cite{kjaer_v1} that the H-spaces under consideration in this article are $\Phi_1$-good.
This paper's second result (Theorem~\ref{thm:diagram_of_fibre_seqs}) is an alternative proof of this fact using our presentation of the K-theory cochains.
Kjaer's original proof makes use of a computation of the homotopy groups of $\Phi_1 X$, and of the homotopy groups of the TAQ-cohomology of $\KU_p^{X_+}$ by means of a spectral sequence.
Our proof on the other hand avoids a computation of these homotopy groups and stays in the realm of higher algebra throughout.
This relies on the fact that TAQ-cohomology sends a presentation of an augmented algebra to a fibre sequence of spectra.

\subsection{A presentation of K-theory cochains}
One way to understand an algebra $R$ over a field $k$ is to give a description in terms of generators and relations, i.e., an isomorphism of the form
\[
    R \cong k[x_1,\dotsc,x_n]/(f_1,\dotsc,f_m),
\]
where $f_i \in k[x_1,\dotsc,x_n]$.
When working with ring spectra, there is no longer a good notion of elements of an algebra.
One instead has to rewrite everything more categorically in terms of morphisms alone.
In the algebraic case, we can reformulate the above isomorphism by saying we have a pushout square of $k$-algebras
\[
    \begin{tikzcd}
        k[y_1,\dotsc,y_m] \arrow[r,"y_i\shortmapsto f_i"] \arrow[d] &  k[x_1,\dotsc,x_n] \arrow[d] \\
        k \arrow[r] & k[x_1,\dotsc,x_n]/(f_1,\dotsc,f_m). \arrow[ul, phantom, "\ulcorner", very near start]
    \end{tikzcd}
\]

Taking this concept to higher algebra, let $A$ be an $\E_\infty$-ring spectrum.
The analogue of a polynomial algebra is the \emph{symmetric $A$-algebra} $\Sym_A(M)$ on an $A$-module spectrum $M$.
This is the $\E_\infty$-$A$-algebra with the universal property that
\[
    \Map_{\CAlg_A}(\Sym_A(M),\ B) \simeq \Map_{\Mod_A}(M,\ B)
\]
for any $\E_\infty$-$A$-algebra $B$.

Let $B$ be an $\E_\infty$-$A$-algebra.
A \emph{presentation} of $B$ is a (homotopy) pushout square of $\E_\infty$-$A$-algebras
\[
    \begin{tikzcd}
        \Sym_A(M) \arrow[r] \arrow[d] & \Sym_A(N) \arrow[d] \\
        A \arrow[r] & B. \arrow[ul, phantom, "\ulcorner", very near start]
    \end{tikzcd}
\]
The module $N$ plays the role of the `generators' of $B$, and the map $\Sym_B(M) \to \Sym_B(N)$ sends $M$ to the `relations' between these generators.
Informally, the cofibre sequence is a witness to $B$ being obtained by imposing additional relations in the symmetric algebra $\Sym_A(N)$ by gluing in a cell for each relation.
If we work with augmented $A$-algebras, which is the case in this paper, then we can equivalently formulate a presentation as a cofibre sequence of augmented $\E_\infty$-$A$-algebras
\[
    \Sym_A(M) \to \Sym_A(N) \to B,
\]
since $A$ is a zero object in this $\infty$-category.

One can also work in an $E$-local setting, where $E$ is a spectrum.
In that case, we instead look for a cofibre sequence
\[
    \mathop{L_E} \Sym_A(M) \to  \mathop{L_E} \Sym_A(N) \to B
\]
of \emph{$E$-local} augmented $\E_\infty$-$A$-algebras, meaning that the $E$-localisation of the ordinary cofibre of $\mathop{L_E} \Sym_A(M) \to \mathop{L_E} \Sym_A(N)$ is identified with $B$.
In this paper, we shall work in a $\K(1)$-local setting throughout.
Because of this, it will be convenient to abbreviate
\[
    Y \hotimes Z := \mathop{L_{\K(1)}}\br*{Y\otimes Z}
\]
for spectra $Y,Z$.

In §\ref{sec:construction_presentation}, we construct a presentation of the following form; see \cref{constr:R_and_T} for the definition of $R$ and $T$.
\begin{theoremLetter}
\label{thm:cochain_presentation}
Let $p$ be an odd prime.
Let $X$ be a pointed space such that we have an isomorphism $\KU_p^*(X) \cong \Lambda_{(\KU_p)_*}[G]$ of $\theta$-algebras with Adams operations, where $G$ is a Morava module concentrated in odd degree which is finitely generated and free as a $(\KU_p)_*$-module; see Notation~{\upshape\ref{not:ext_algebra}}.
Then there exists a cofibre sequence in the $\infty$-category $\mathop{L_{\K(1)}} \CAlg^\aug_{\KU_p}$ of the form
\[
    \begin{tikzcd}
        \KU_p \hotimes \Sym(M(G)) \arrow[r, "R"] & \KU_p \hotimes \Sym(M(G)) \arrow[r, "T"] & \KU_p^{X_+}.
    \end{tikzcd}
\]
\end{theoremLetter}

The proof requires a calculation of the homotopy groups of a relative $\K(1)$-local tensor product involving free algebras of the above form.
For this, we use a $\K(1)$-local version of the Tor spectral sequence that we develop in Appendix~\ref{app:K1_Tor_SS}, which is joint work with Max Blans.
Its $E^2$-page consists of certain \emph{completed} Tor groups, which are more computable in this case; see Theorem~\ref{prop:K1_local_Tor_SS} in Appendix~\ref{app:K1_Tor_SS} for the precise statement.
We do not know how to compute the $E^2$-page of the ordinary Tor spectral sequence in this setting; see \cref{rmk:why_use_decompletions} for a further discussion.

\begin{remark}
Roughly speaking, the isomorphism $\KU_p^*(X) \cong \Lambda[G]$ is saying that $\KU_p^*(X)$ is an exterior algebra on a finite number of odd generators, where these generators have a specified action by the $p$-adic Adams operations.
The spectrum $M(G)$ is constructed in §\ref{ssec:Ktheory_Moore_spectra} and is determined by being the $\K(1)$-local spectrum with the property that
\[
    \compKU(M(G)) := \pi_*( \KU_p \hotimes M(G) ) \cong G.
\]
\end{remark}

\begin{remark}
The maps $R$ and $T$ admit the following descriptions.
The ring $\KU_p^*(X)$ is a \emph{$\theta$-algebra}, meaning it carries a non-linear operation $\theta^p$.
The isomorphism of $\theta$-algebras $\KU_p^*(X) \cong \Lambda[G]$ in particular means (see Notation~\ref{not:ext_algebra}) that the $\theta^p$-operation on $\KU_p^*(X)$ comes from a $\Z_p$-linear map $\theta^p_G \colon G \to G$.
By a result of McClure, we have an isomorphism
\[
    \compKU(\Sym(M(G))) \cong \Free_\theta[G],
\]
where $\Free_\theta[G]$ is the \emph{free $\theta$-algebra} on $G$, which is the ring obtained by adding formal iterates of a formal operation $\theta^p$ to $G$.
The map $R$ corresponds to a map
\[
    \KU_p \hotimes M(G) \to \KU_p \hotimes \Sym(M(G))
\]
which on homotopy groups is the map $G \to \Free_\theta[G]$ sending $y$ to $\theta^p (y) - \theta^p_G(y)$.
Informally, $R$ encodes the relation that the formal operation $\theta^p$ must coincide with the concrete operation $\theta^p_G$ on $G$.
The map $T$ corresponds to a map
\[
    \KU_p \hotimes M(G) \to \KU_p^{X_+}
\]
which on homotopy groups is the inclusion $G \to \Lambda[G]$.
Informally, $T$ sends the generators of the symmetric algebra to the concrete generators of the cochains.
\end{remark}

For some $X$, we can find explicit expressions for $M(G)$, $R$, and $T$.

\begin{example}
Consider $X = S^{2n+1}$ an odd sphere, where $n\geq 0$.
In this case, writing $G$ for the Morava module
\[
    G := \redKU_p^*(S^{2n+1}) \cong \compKU(\S^{-2n-1}),
\]
we have an isomorphism of graded rings
\[
    \KU_p^*(X) \cong \Lambda[G].
\]
The operation $\theta^p$ on $\KU_p^*(X)$ is given by multiplication by $p^n$.
Writing $\theta^p_G$ for the map $G \to G$ given by multiplication by $p^n$, this induces a $\theta$-algebra structure on $\Lambda[G]$, making the above ring isomorphism even one of $\theta$-algebras.
The spectrum $M(G)$ is $\S^{-2n-1}_{\K(1)}$.
As such, the resulting presentation for the cochains on $S^{2n+1}$ is of the form
\[
    \begin{tikzcd}
        \mathop{L_{\K(1)}}\KU_p\{y\} \arrow[r, "y \ \mapsto\ \theta^p (x) - p^n x"] &[5em] \mathop{L_{\K(1)}} \KU_p\{x\} \arrow[r] & \KU_p^{S^{2n+1}_+}
    \end{tikzcd}
\]
with both $x$ and $y$ in degree $-2n-1$.
Here $\KU_p\{x\}$ denotes the symmetric $\KU_p$-algebra on one generator, whose fundamental class is labelled $x$.
\end{example}

Bousfield \cite[Thm.~6.3]{Bous_v1_Hspaces} showed that many H-spaces have $p$-adic K-theory as described in Theorem~\ref{thm:cochain_presentation}: if $X$ is a simply connected H-space such that $\uH_*(X;\Q)$ is associative and $\uH_*(X;\Z_{(p)})$ is finitely generated over $\Z_{(p)}$ (in particular, it vanishes above some degree), then $\KU_p^*(X)$ is of the form described in Theorem~\ref{thm:cochain_presentation}, with an explicit $G$.
Note that this in particular applies to all simply connected compact Lie groups.

\begin{example}
Consider $X = \SU(n)$, where $n\geq 2$.
The relevant module $G$ in this case is isomorphic to $\redKU_p^*(\CP^{n-1})$; see \cite[Ex.~9.3]{Bous_v1_Hspaces}.
Recall that $\KU_p^*(\CP^{n-1})$ is a truncated polynomial ring $\Z_p[x]/(x^n)$, where $x+1$ is the class of the tautological line bundle on $\CP^{n-1}$.
From this, we see that the Adams operation $\psi^k$ (for $k \in \N$) and the operation $\theta^p_G$ are given by
\begin{align*}
    \SwapAboveDisplaySkip
    \psi^k(x) &= \sum_{j=1}^k \binom{k}{j}x^j;\\
    \theta^p_G(x) &= \frac{1}{p}\sum_{j=1}^{p-1} \binom{p}{j} x^j.
\end{align*}
The spectrum $M(G)$ is the $\K(1)$-local Spanier--Whitehead dual of $\Sigma^\infty \CP^{n-1}$.
\end{example}

\subsection{Comparing cochains with the Bousfield--Kuhn functor}
Let $n\geq 1$, and let $\Phi_n\colon \Spaces_* \to \Sp_{\mathrm{T}(n)}$ denote the height~$n$ telescopic Bousfield--Kuhn functor \cite[§5]{Behrens_Rezk_survey_Kn}.
Behrens and Rezk \cite[§6]{Behrens_Rezk_Kn} construct a \emph{comparison map}
\[
    c_X \colon \mathop{ L_{\K(n)}}\Phi_n X \to \TAQ_{\S_{\K(n)}}(\S_{\K(n)}^{X_+})^\vee
\]
where $\TAQ(-)^\vee$ denotes topological Andr\'e--Quillen cohomology.
If this comparison map is an equivalence, then we think of the cochains as providing a good model for the $v_n$-periodic homotopy type of $X$.
This heuristic is made precise in the work of Heuts \cite{Heuts_Lie}: the spectrum $\Phi_n X$, when endowed with the structure of a spectral Lie algebra, \emph{does} capture all $v_n$-periodic information about $X$.
If the spectrum $L_{\K(n)}\Sigma^\infty_+X$ is $\K(n)$-locally dualisable, then the map $c_X$ is an equivalence if and only if $X$ is $L_{\K(n)}\Phi_n$-good (i.e., the Goodwillie tower for $L_{\K(n)} \Phi_n$ converges at $X$); see \cite[Cor.~7.15]{Heuts_handbook_Lie}.

Let us now, and henceforth, specialise to the height $n=1$ case.
In this case, the telescope conjecture is a theorem, so that $\mathrm{T}(1)$-local spectra are the same as $\K(1)$-local spectra, and hence $L_{\K(1)}\Phi_1 \simeq \Phi_1$; see, e.g., \cite[§3]{barthelShortIntroductionTelescope2020} for an overview.
Even at height~$1$, the comparison map $c_X$ can be hard to work with, because it is defined over the $\K(1)$-local sphere.
There is a $\KU_p$-linear version of the comparison map
\[
    c_X^{\KU_p} \colon \KU_p \hotimes \Phi_1 X \to \TAQ_{\KU_p}(\KU_p^{X_+})^\vee,
\]
which lends itself more to computation.
Under suitable further finiteness assumptions (see, e.g., Lemma~\ref{lem:Phi_good_iff_K_theoretic_equiv}), the integral comparison map $c_X$ is an equivalence if and only if the $\KU_p$-linear one is.

The presentation from Theorem~\ref{thm:cochain_presentation}, combined with work by Bousfield, can be used to show that many of the aforementioned H-spaces are $\Phi_1$-good.
More specifically, for $X$ satisfying the conditions of Theorem~\ref{thm:cochain_presentation} plus additional conditions (see Notation~\ref{not:ext_algebra_plus_Bousfield_stuff}), Bousfield \cite[Thm.~8.1]{Bous_v1_Hspaces} computed $\Phi_1X$, showing that there is a fibre sequence
\[
    \Phi_1X \to M^\vee(G) \to M^\vee(G),
\]
where $M^\vee(G)$ is the $\K(1)$-local Spanier--Whitehead dual of $M(G)$.
Applying TAQ-cohomology (see §\ref{ssec:TAQ} for a brief review) to the cofibre sequence of Theorem~\ref{thm:cochain_presentation} yields a fibre sequence
\[
    \TAQ_{\KU_p}(\KU_p^{X_+})^\vee \to \map(M(G),\ \KU_p) \to \map(M(G),\ \KU_p).
\]
In §\ref{ssec:construction_of_diagram}, we compare these fibre sequences by constructing a diagram of the following form; see \cref{constr:big_diagram_Thm_B} for the construction of the diagram.

\begin{theoremLetter}
\label{thm:diagram_of_fibre_seqs}
Let $X$ be a pointed space satisfying the conditions of Notation~{\upshape\ref{not:ext_algebra_plus_Bousfield_stuff}}.
Then there exists a commutative diagram in the $\infty$-category $\mathop{L_{\K(1)}}\Mod_{\KU_p}$ of the form
\[
    \begin{tikzcd}
        \KU_p \hotimes \Phi_1 X \arrow[r] \arrow[d,"c_X^{\KU_p}"] &[1.5em] \KU_p \hotimes M^\vee(G) \arrow[r] \arrow[d,"\simeq"] &[1.5em] \KU_p \hotimes M^\vee(G) \arrow[d,"\simeq"]\\
        \TAQ_{\KU_p}(\KU_p^{X_+})^\vee \arrow[r,"\TAQ(T)^\vee"] & \map(M(G),\ \KU_p) \arrow[r,"\TAQ(R)^\vee"] & \map(M(G),\ \KU_p),
    \end{tikzcd}
\]
with the top sequence the {\upshape(}$\KU_p$-linearised{\upshape)} fibre sequence from Bousfield, the bottom sequence the fibre sequence coming from the presentation of Theorem~{\upshape\ref{thm:cochain_presentation}}, and with $c_X^{\KU_p}$ the $\KU_p$-linear Behrens--Rezk comparison map.
\end{theoremLetter}

Thus $c_X^{\KU_p} $ is an equivalence for such $X$.
As detailed in §\ref{ssec:construction_of_diagram}, this implies that $c_X$ is an equivalence, and thus recovers the following result of Kjaer.

\begin{corollaryLetter}[\cite{kjaer_v1}]
Let $p$ be an odd prime.
If $X$ is a pointed space satisfying the conditions of Notation~{\upshape\ref{not:ext_algebra_plus_Bousfield_stuff}}, then $X$ is $\Phi_1$-good.
In particular, if $X$ is a simply connected H-space such that $\uH_*(X;\Q)$ is associative and $\uH_*(X;\Z_{(p)})$ is a finitely generated $\Z_{(p)}$-module, then $X$ is $\Phi_1$-good.
\end{corollaryLetter}

\subsection{Conventions and notation}
Throughout this paper, we fix an odd prime $p$.

We use the language of $\infty$-categories in the sense of Lurie \cite{HTT} as our formalism for homotopy theory.
This is merely preferential, and the proofs and constructions should carry over to a different system without much difficulty.
All limits and colimits should be understood as \emph{homotopy} limits and colimits.

We view K-theory homology and cohomology groups as $\Z/2$-graded, by recording only the K-homology groups in degree $0$ and $1$ (or $0$ and $-1$ for K-cohomology).
In particular, we make the identifications
\[
    \KU_* \cong \Z \qquad \text{and} \qquad (\KU_p)_* \cong \Z_p,
\]
both concentrated in degree 0.

The smash product of two spectra $Y$ and $Z$ is denoted by $Y \otimes Z$, and the mapping spectrum from $Y$ to $Z$ is denoted by $\map(Y,Z)$.
We write $\hotimes$ for the $\K(1)$-local smash product of $Y$ and $Z$, i.e.,
\[
    Y \hotimes Z := \mathop{L_{\K(1)}}\br*{Y\otimes Z}.
\]
We write $\compKU(Y) := \pi_*(\KU_p\hotimes Y)$ for the completed $\KU_p$-homology of $Y$.

We write $\Mod^*_{\Z_p}$ for the category of $\Z/2$-graded $\Z_p$-modules, and $\CAlg^*_{\Z_p}$ for the category of graded-commutative $\Z/2$-graded $\Z_p$-modules.
We write $\Modh^*_{\Z_p}$ for the full subcategory of $\Mod_{\Z_p}^*$ on the derived $p$-complete modules (a.k.a.\ $L$-complete modules), and $\CAlgh^*_{\Z_p}$ for the full subcategory of $\CAlg_{\Z_p}^*$ on the algebras whose underlying module is derived $p$-complete.
(See §\ref{ssec:morava_modules} for a brief reminder of derived $p$-completeness.)

In Appendix~\ref{app:K1_Tor_SS} we work with $\Z$-graded objects rather than $\Z/2$-graded objects.
As such, in the appendix (and only in the appendix) the notation $\Mod_{\Z_p}^*$ stands for $\Z$-graded $\Z_p$-modules, rather than $\Z/2$-graded $\Z_p$-modules.

\subsection*{Acknowledgements}
\addcontentsline{toc}{subsection}{\protect\numberline{}Acknowledgements}
The present work is a product of my Master's thesis.
First and foremost, I would like to thank Gijs Heuts for suggesting the topic, for his guidance, and for being ever ready to answer my many questions.
I have learned much from my conversations with Max Blans on the topic, and am particularly thankful for his suggestion of using a $\K(1)$-local Tor spectral sequence, which now appears in our joint appendix to this paper.
I thank Jack Davies, Gijs Heuts and Lennart Meier for their proofreading and suggestions, and Lennart in particular for catching a subtle mistake in an earlier version.
Finally, I would like to thank the anonymous referee for helpful comments and suggestions.

Many of the ideas presented here originate from the paper by Bousfield \cite{Bous_v1_Hspaces}.

\section{K-theoretic preliminaries}
\label{sec:Kthy_preliminaries}
In this section we review some standard results that we need in the remainder of this text.
In §\ref{ssec:higher_algebra} we review two basic constructions from higher algebra: symmetric $A$-algebras and $A$-cochains (where $A$ is an $\E_\infty$-ring spectrum).
The remainder of the section is devoted to $p$-completed K-theory.
In §\ref{ssec:morava_modules} we discuss the structure that the Adams operations give to the K-theory of a spectrum, turning it into a Morava module.
In §\ref{ssec:theta_algebras} we discuss $\theta$-algebras, which is the additional structure on the K-homology of an $\E_\infty$-ring spectrum and the K-cohomology of a space.
Lastly, in §\ref{ssec:Ktheory_Moore_spectra} we construct a K-theory Moore spectrum $M(G)$ on a Morava module $G$ that is finitely generated and concentrated in either even or odd degree.

\subsection{Recollections from higher algebra}
\label{ssec:higher_algebra}

\begin{theorem}[\cite{HA}, Prop.~3.1.3.13]
\label{thm:free_alg_in_symm_mon}
Let $\calC$ be a symmetric monoidal $\infty$-category.
Suppose that $\calC$ has all countable colimits and that the tensor product functor preserves these in each variable separately.
Then the forgetful functor $\CAlg(\calC)\to \calC$ admits a left adjoint.
\end{theorem}

We will denote such a left adjoint by $\Sym_\calC \colon \calC \to \CAlg(\calC)$ and call it the \defi{symmetric algebra functor}.

If $0$ is an initial object of $\calC$, then $\Sym(0)$ is an initial object of $\CAlg(\calC)$, i.e., it is the unit algebra $\mathbf{1}$.
If in addition $0$ is also a terminal object of $\calC$ (in other words, it is a zero object), then by taking slices over $0$, the functor $\Sym_\calC$ lifts to a functor $\calC \to \CAlg^\aug(\calC)$.

\begin{example}
Let $A$ be an $\E_\infty$-ring spectrum.
We write $\Mod_A$ for $\Mod_A(\Sp)$, and $\CAlg_A$ for $\CAlg_A(\Sp)$.
As $\Sp$ has all small colimits, so does $\Mod_A$ by \cite[Cor.~4.2.3.5]{HA}, and so the above theorem yields a functor $\Sym_A \colon \Mod_A \to \CAlg_A^\aug$.
By \cite[Prop.~3.1.3.13]{HA}, the underlying spectrum of $\Sym_A(M)$ is given by
\[
    \bigoplus_{n\geq 0} (M^{\otimes_A n})_{h\Sigma_n}.
\]
Through \cite[Thm.~4.5.4.7]{HA}, this retrieves the definitions in a model for spectra of \cite[§§II.4, III.5]{EKMM}.

If $A \to B$ is a map of $\E_\infty$-ring spectra, then we have a natural equivalence (where $M$ is an $A$-module spectrum)
\[
    B \otimes_A \Sym_A(M) \simeq \Sym_B(B\otimes_A M).
\]
In the absolute case where $A$ is the sphere spectrum $\S$, we omit the subscript and simply write $\Sym$ for $\Sym_\S$.
\end{example}

\begin{example}
Let $E$ be a spectrum, and let $A$ be an $E$-local $\E_\infty$-ring spectrum.
We write $\mathop{L_E} \Mod_A$ for the $\infty$-category of $E$-local $A$-modules.
The symmetric algebra functor for the $\infty$-category $\calC = \mathop{L_E} \Mod_A$ is equivalent to the composite $\mathop{L_E} \Sym_A$ of $E$-localisation with the symmetric $A$-algebra functor.
Indeed, $\mathop{L_E} \Sym_A$ is left adjoint to the forgetful functor $\mathop{L_E} \CAlg_A \to \mathop{L_E}\Mod_A$.
\end{example}

We shall apply the above mostly in the case where $A = \KU_p$ and $E = \K(1)$.
In particular, we shall use the natural equivalence (where $Y$ is a spectrum)
\[
    \KU_p \hotimes \Sym(Y) \simeq \mathop{L_{\K(1)}} \Sym_{\KU_p}(\KU_p \hotimes Y).
\]

\begin{definition}
Let $\calC$ be a symmetric monoidal $\infty$-category.
Suppose that $\calC$ has all small limits.
Let $A \in \CAlg(\calC)$.
Recall from \cite[Thm.~5.1.5.6]{HTT} that the $\infty$-category of spaces $\Spaces$ is freely generated under small colimits by a point.
Using this, we define a functor
\[
    A^{(-)_+} \colon \Spaces^\op \to \CAlg_A(\calC)
\]
by requiring it to preserve small limits and to send a point to $A$.
If $X$ is a space, we call $A^{X_+}$ the \defi{$A$-cochains} of $X$.
\end{definition}

Taking slices under a point, the cochains functor also has a pointed version landing in augmented algebras:
\[
    A^{(-)_+} \colon \Spaces_*^\op \to \CAlg^\aug_A(\calC).
\]

\begin{proposition}
Let $\calC$ be a closed symmetric monoidal $\infty$-category.
Suppose that $\calC$ has all small limits and colimits.
Let $A \in \CAlg(\calC)$.
Then the composite functor
\[
    \begin{tikzcd}
        \Spaces^\op \arrow[r,"A^{(-)_+}"] & \CAlg_A(\calC) \arrow[r] & \Mod_A(\calC)
    \end{tikzcd}
\]
is equivalent to the functor $X\mapsto \map_{\Mod_A(\calC)}(A\otimes X,\ A)$, where $A \otimes X = \colim_X A$ denotes the tensoring of $\Mod_A(\calC)$ over $\Spaces$.
\end{proposition}
\begin{proof}
The functor $\CAlg_A(\calC) \to \Mod_A(\calC)$ preserves all small limits by \cite[Cor.~3.2.2.5]{HA}.
Therefore both $A^{(-)_+}$ and $\map_{\Mod_A(\calC)}(A\otimes -,\ A)$ are limit-preserving functors $\Spaces^\op \to \Mod_A(\calC)$ that send a point to $A$.
\end{proof}

\begin{example}
If $A$ is an $\E_\infty$-ring spectrum and $X$ is a space, then $A^{X_+}$ is an $\E_\infty$-$A$-algebra spectrum refining the mapping spectrum $\map(\Sigma^\infty_+X,\ A)$.
We in particular obtain an isomorphism
\[
    \mathop{\pi_{-*}} A^{X_+} \cong A^*(X).
\]
If $E$ is a spectrum and $A$ is $E$-local, then $A^{X_+}$ is $E$-local for every space $X$.
Indeed, the forgetful functor $\CAlg_A \to \Mod_A$ preserves limits by \cite[Cor.~3.2.2.5]{HA}, and $E$-local modules are closed under limits in $\Mod_A$.
\end{example}

\subsection{Morava modules}
\label{ssec:morava_modules}
We will need the notion of \emph{derived $p$-completion} (also called \emph{$L$-completion}) of $\Z/2$-graded $\Z_p$-modules.
The functor of ordinary $p$-completion $(-)_p^\wedge$ on $\Mod_{\Z_p}^*$ has two left derived functors, denoted $L_0$ and $L_1$.
Here $L_0$ is right exact and $L_1$ is left exact.
We call $L_0 M$ the \defi{derived $p$-completion} of $M$.
In case $M$ has bounded $p$-torsion, this is naturally isomorphic to $M_p^\wedge$.
A module $M$ is called \defi{derived $p$-complete} if the natural map $M \to L_0 M$ is an isomorphism.
We refer to Hovey and Strickland \cite[App.~A]{hovey_strickland_Ktheories_localisation} and Barthel and Frankland \cite[App.~A]{barthel_frankland_power_op_MoravaE} for an introduction to the theory of derived $p$-completion.
We write
\[
    \Modh^*_{\Z_p} \qquad \text{and}\qquad \CAlgh^*_{\Z_p}
\]
for the full subcategories of $\Mod_{\Z_p}^*$ and $\CAlg_{\Z_p}^*$ on the derived $p$-complete modules and algebras, respectively.

\begin{definition}[\cite{goerss_hopkins_moduli_problems_structured}, Def.~2.2.1]
A \defi{$p$-adic Morava module} is a derived $p$-complete topological $\Z/2$-graded $\Z_p$-module $M$ with a continuous action of $\Z_p^\times$ by degree-preserving maps, such that the quotient $M/p$ is a discrete $\Z_p^\times$-module.
A \defi{morphism} of Morava modules is a continuous morphism of graded $\Z_p$-modules that intertwines the action of $\Z_p^\times$.
We write $\MorMod_p$ for the category of $p$-adic Morava modules.
\end{definition}

Often we will write the action of $k \in \Z_p^\times$ on a Morava module $M$ as $\psi^k$ (or by $\psi^k_M$ if we wish to emphasise the module $M$) and call these the \defi{Adams operations} of $M$.
Often we will also omit the prime $p$ and refer to $p$-adic Morava modules simply as \emph{Morava modules}.

\begin{remark}
In \cite{Bous_v1_Hspaces}, Bousfield uses the term \emph{stable Adams module} (see Definition~2.6 of op.\ cit.) for a closely related concept.
If $M$ is finitely generated as a $\Z_p$-module, then the concept coincides with that of a $p$-adic Morava module.
Bousfield uses the term \emph{Adams module} to mean a stable Adams module which also has a compatible $\psi^p$ operation.
\end{remark}

\begin{proposition}
\label{prop:Kthy_of_spectrum_is_Morava}
If $X$ is a $\K(1)$-local $\KU_p$-module spectrum, then $\pi_*X$ is naturally a $p$-adic Morava module.
\end{proposition}
\begin{proof}
Goerss and Hopkins \cite[Prop.~2.2.2]{goerss_hopkins_moduli_problems_structured} give a Morava module structure to the completed $\KU_p$-homology $\compKU(X) = \pi_*(\KU_p\hotimes X)$ of a spectrum $X$, natural in maps of spectra.
Using the same reasoning as in \cite[Prop.~6.8]{barthel_frankland_power_op_MoravaE}, this induces Adams operations on $\pi_* M$ for any $\K(1)$-local $\KU_p$-module spectrum.
\end{proof}

\subsection{\texorpdfstring{$\theta$}{Theta}-algebras}
\label{ssec:theta_algebras}

\begin{definition}[\cite{goerss_hopkins_moduli_problems_structured}, Def.~2.2.3]
\label{def:theta_alg}
A \defi{$p$-adic $\theta$-algebra with Adams operations} is a graded-commutative $\Z/2$-graded $\Z_p$-algebra $R$ with the structure of a Morava module, together with continuous maps
\[
    \theta^p \colon R_0 \to R_0 \qquad \text{and} \qquad \theta^p \colon R_1 \to R_1,
\]
satisfying the following conditions.
\begin{condenum}
    \item For all $k\in \Z_p^\times$, the operation $\psi^k$ on $R$ is $\Z_p$-linear, and if $x,y\in R$ are homogeneous, then
    \[
        \psi^k(xy) = \begin{cases}
        \ \psi^k(x)\cdot \psi^k(y) &\abs{x} = 0 \quad \text{or}\quad \abs{y} = 0,\\
        \ \frac{1}{k}\cdot \psi^k(x) \cdot \psi^k(y) &\abs{x} = \abs{y} = 1.
        \end{cases}
    \]
    \item We have $\theta^p (1) = 0$.
    \item For all $k\in \Z_p^\times$, we have $\theta^p \circ \psi^k = \psi^k \circ \theta^p$.
    \item\label{cond:additivity_theta} If $x,y\in R$ are homogeneous, then
    \[
        \theta^p(x+y) = \begin{cases}
        \ \theta^p(x) + \theta^p(y) - \frac{1}{p}\sum_{i=1}^{p-1} \binom{p}{i} \cdot x^iy^{p-i} &\abs{x} = \abs{y} = 0,\\
        \ \theta^p(x) + \theta^p(y) &\abs{x} = \abs{y} = 1.
        \end{cases}
    \]
    \item If $x,y\in R$ are homogeneous, then
    \[
        \theta^p(xy) = \begin{cases}
        \ \theta^p(x) \cdot y^p + x^p \cdot \theta^p(y) + p\cdot \theta^p(x)\cdot \theta^p(y) &\abs{x} = 0 \quad \text{or}\quad \abs{y} = 0,\\
        \ \theta^p(x) \cdot \theta^p(y) &\abs{x} = \abs{y} = 1.
        \end{cases}
    \]
\end{condenum}
A \defi{morphism} of $p$-adic $\theta$-algebras with Adams operations is a morphism of graded $\Z_p$-algebras that is a morphism of Morava modules and that intertwines the action of $\theta^p$.
We write $\Alg_{\theta,p}$ for the category of $p$-adic $\theta$-algebras with Adams operations.
\end{definition}

As with Morava modules, often we will omit the prime $p$ and refer to these algebras as \emph{$\theta$-algebras with Adams operations}, or sometimes even simply as \emph{$\theta$-algebras}.
In that case we may also write $\theta$ for $\theta^p$.

We extend $\theta^p$ to a map $R \to R$ by putting $\theta^p(x+y) = \theta^p(x) + \theta^p(y)$ for $x$ even and $y$ odd.
We define $\psi^p \colon R \to R$ by putting $\psi^p(x) := x^p + p \cdot \theta^p(x)$.
If $R$ has no $p$-torsion, then $\theta^p$ can be recovered from $\psi^p$.
The conditions on $\theta^p$ ensure that $\psi^p$ is an additive operation satisfying a multiplicative law given similar to (i) above: if $x,y\in R$ are homogeneous, then
\[
    \psi^p(x)\cdot \psi^p(y) = \begin{cases}
        \ \psi^p(xy) &\abs{x} = 0 \quad \text{or}\quad \abs{y} = 0,\\
        \ p\cdot \psi^p(xy) &\abs{x} = \abs{y} = 1.
    \end{cases}
\]
With this notation, we can rewrite part of condition (v) above as stating that
\begin{equation}
    \label{eq:multiplicativity_theta}
    \theta^p(xy) = \psi^p(x) \cdot \theta^p(y)
\end{equation}
when $x$ is even and $y$ is odd.

\begin{remark}
\label{rmk:bousfield_notation_psi_theta}
Some denote the map $\theta^p \colon R_1 \to R_1$ by either $\psi^p$ or $\psi$; note that these differ from $\psi^p\colon R_1 \to R_1$ as defined above by a factor of $p$.
See, e.g., \cite{Bous_Ktheory_infloop, Bous_Ktheory_Hspaces, Bous_v1_Hspaces, barthel_frankland_power_op_MoravaE}.
\end{remark}

\begin{remark}
\label{rmk:difference_bousfield_meaning_padic}
Bousfield uses the term \emph{$p$-adic $\theta$-algebra} to mean something slightly different than the above, requiring $p$-completeness in a stronger sense than derived $p$-completeness.
See \cite[§6]{Bous_Ktheory_infloop} or \cite[§1]{Bous_Ktheory_Hspaces} for precise definitions.
Every $p$-adic $\theta$-algebra with Adams operations in the sense of Bousfield determines a $\theta$-algebra in the sense of Definition~\ref{def:theta_alg}.
\end{remark}

\begin{example}
The ring $\Z_p$ can be given the structure of a $\theta$-algebra with Adams operations by putting $\psi^k$ equal to the identity for all $k\in\Z_p$.
This determines $\theta^p$ as well, due to the absence of $p$-torsion.
We shall always regard $\Z_p$ as having these Adams operations.
\end{example}

\begin{remark}
\label{rmk:theta_linear_on_odd_elements}
If $R$ is a $\theta$-algebra augmented over $\Z_p$, then $\theta^p$ is $\Z_p$-linear on odd elements of $R$: this follows from Equation~\eqref{eq:multiplicativity_theta}.
\end{remark}

\begin{theorem}
If $A$ is a $\K(1)$-local $\mathbf{H}_\infty$-$\KU_p$-algebra spectrum, then $\pi_*A$ is naturally a $p$-adic $\theta$-algebra with Adams operations.
\end{theorem}
\begin{proof}
This follows from the work by McClure \cite[Ch.~IX]{H_infty_book}, as explained by Barthel and Frankland \cite[§6]{barthel_frankland_power_op_MoravaE}.
Specifically, in Proposition~6.8 of op.\ cit.\ they show that McClure's operation $Q$ extends to an operation on all $\K(1)$-local $\mathbf{H}_\infty$-$\KU_p$-algebras.
Then in §6.3, they show that this operation $Q$ satisfies the conditions on $\theta^p$ in a $\theta$-algebra.
The Adams operations from Proposition~\ref{prop:Kthy_of_spectrum_is_Morava} above commute with $Q$, because the Adams operations on $\KU_p$ are $\E_\infty$-ring maps.
This gives $\pi_* A$ the structure of a $p$-adic $\theta$-algebra with Adams operations in the sense of Definition~\ref{def:theta_alg}.
\end{proof}

\begin{corollary}
\leavevmode 
\begin{thmenum}
    \item If $A$ is an $\E_\infty$-ring spectrum, then $\compKU(A)$ is naturally a $p$-adic $\theta$-algebra with Adams operations.
    \item If $X$ is a space, then $\KU_p^*(X)$ is naturally a $p$-adic $\theta$-algebra with Adams operations.
\end{thmenum}
\end{corollary}
\begin{proof}
If $A$ is an $\E_\infty$-ring spectrum, then $\KU_p\hotimes A$ is a $\K(1)$-local $\E_\infty$-$\KU_p$-algebra.
If $X$ is a space, then $\KU_p^{X_+}$ is an $\E_\infty$-$\KU_p$-algebra, and it is $\K(1)$-local because $\KU_p$ is.
\end{proof}

The forgetful functor $\Alg_{\theta,p} \to \MorMod_p$ admits a left adjoint $\Free_{\theta,p}$, the \defi{free $\theta$-algebra} functor.
We will only construct this free $\theta$-algebra in the following special case, which is all we need.
\begin{example}
\label{ex:free_theta_alg_odd_generators}
Let $G$ be a Morava module concentrated in odd degree.
We can form $\Free_\theta[G]$ as follows.
Let $FG$ denote the module $G\oplus G \oplus \dotsb$, which inherits Adams operations from $G$.
We topologise $FG$ by giving it the coproduct topology; see \cite[§2]{higginsCoproductsTopologicalAbelian1977}.
(Note that $FG$ is not yet a Morava module, as it is not derived $p$-complete when $G$ is nonzero.)
Define a $\Z_p$-linear homomorphism $\theta^p \colon FG \to FG$ by shifting each copy of $G$ one to the right.
Now define as algebras
\[
    \Free_\theta[G] := \mathop{L_0} \Lambda_{\Z_p}[FG].
\]
The Adams operations $\psi^k$ are induced from the Adams operations on $FG$.
The homomorphism $\theta^p$ on $FG$ uniquely determines an operation $\theta^p$ on $\Free_\theta[G]$ that turns it into a $p$-adic $\theta$-algebra.

Let $R$ be a $\theta$-algebra and $f\colon G \to R$ a morphism of Morava modules.
Then the induced morphism $\Free_\theta[G] \to R$ can be described as follows.
First define a map $Ff\colon FG \to R$ given by applying the map $(\theta^p_R)^{\circ t} \circ f$ to the $t$-th copy of $G$.
Note that $Ff$ is a $\Z_p$-module homomorphism that intertwines the Adams operations.
The map $Ff$ induces a map $\mathop{L_0} \Lambda [FG] \to R$, which is the map $\Free_\theta[G] \to R$.
\end{example}

\begin{remark}
Due to Remark~\ref{rmk:difference_bousfield_meaning_padic}, the free $\theta$-algebras introduced by Bousfield differ from the ones of Example~\ref{ex:free_theta_alg_odd_generators}.
See §\ref{ssec:review_Bousfield_v1} for a review of Bousfield's free $\theta$-algebra on a Morava module concentrated in odd degree.
\end{remark}

\begin{theorem}[McClure]
\label{thm:mcclure}
Let $Y$ be a spectrum such that $\compKU(Y)$ is a flat $\Z/2$-graded $\Z_p$-module.
Then the natural map
\[
    \Free_{\theta,p}\sqbr*{\compKU(Y)} \to \compKU(\Sym(Y))
\]
induced by $\compKU(Y) \to \compKU(\Sym(Y))$ is an isomorphism of $p$-adic $\theta$-algebras with Adams operations.
Moreover, this isomorphism is compatible with the free-forgetful adjunctions: if $A$ is an $\E_\infty$-ring spectrum, then the diagram
\[
    \begin{tikzcd}
        \Map_{\CAlg(\Sp)}(\Sym(Y),\ A) \arrow[r,"\simeq"] \arrow[d,"\compKU"'] & \Map_{\Sp}(Y, A) \arrow[d,"\compKU"]\\
        \Hom_\theta\br*{\Free_{\theta,p}\sqbr*{\compKU(Y)},\ \compKU(A)} \arrow[r,"\cong"] & \Hom\br*{\compKU(Y),\ \compKU(A)}
    \end{tikzcd}
\]
commutes up to natural isomorphism.
\end{theorem}
\begin{proof}
This follows from the work by McClure \cite[Ch.~IX]{H_infty_book}, as explained by Barthel and Frankland \cite{barthel_frankland_power_op_MoravaE}.
Specifically, Barthel and Frankland first work with Morava E-theory $\mathrm{E}_h$ at a general height $h\geq 1$.
In §3, they introduce a monad $\hat{\mathbf{T}} = L_0 \bigoplus_{n\geq 0}\hat{\mathbf{T}}_n$.
In Proposition~3.18 they show that the natural map
\[
    \hat{\mathbf{T}}_n(\pi_* M) \to \mathop{\pi_*}(L_{\K(h)}\Sym^n_{\mathrm{E}_h}(M))
\]
is an isomorphism if $\pi_*M$ is flat as a graded module over $\pi_*(\mathrm{E}_h)$.
Then in Theorem~6.14, restricting to height $h=1$, they identify the monad $\hat{\mathbf{T}}$ with the free $\theta$-algebra monad.
\end{proof}

\subsection{K-theory Moore spectra}
\label{ssec:Ktheory_Moore_spectra}
Let $G$ be a $p$-adic Morava module, concentrated either in even or odd degree and which is finitely generated as $\Z_p$-module.
In \cite[Prop.~8.7]{bousf_K_local_spectra_odd_prime}, Bousfield constructs a $\KU$-local spectrum $M_{(p)}(G)$ along with an isomorphism
\[
    (\KU_{(p)})_*(M_{(p)}(G)) \cong G
\]
respecting stable $p$-local Adams operations $\psi^k$ for $k\in\Z_{(p)}^\times$.

\begin{definition}
\label{def:MG}
Let $G$ be a $p$-adic Morava module, concentrated either in even or odd degree and which is finitely generated as $\Z_p$-module.
We write $M(G)$ for the spectrum $M_{(p)}(G)_p^\wedge$.
\end{definition}

Note that $M(G)$ is a $\K(1)$-local spectrum.
As $G$ is derived $p$-complete, the following is immediate.

\begin{lemma}
\label{lem:K_Moore_spectrum}
The isomorphism $(\KU_{(p)})_*(M_{(p)}(G)) \cong G$ induces an isomorphism
\[
    \compKU(M(G)) \cong G
\]
of $p$-adic Morava modules.
\end{lemma}

\begin{remark}
\label{rmk:K_Moore_spectrum_uniqueness}
Up to equivalence, the spectrum $M(G)$ is the unique $\K(1)$-local spectrum with the property of Lemma~\ref{lem:K_Moore_spectrum}.
This follows, e.g., from Proposition~\ref{prop:barthel-heard} below, together with the fact that the functor $\KU_p \hotimes {-}\colon \Sp_{\K(1)} \to \mathop{L_{\K(1)}}\Mod_{\KU_p}$ is conservative.
\end{remark}

\begin{definition}
If $G$ is a Morava module as in Definition~\ref{def:MG}, then we write $M^\vee(G)$ for the $\K(1)$-local Spanier--Whitehead dual of $M(G)$.
\end{definition}

For a Morava module $G$ as above, the spectrum $M(G)$ is $\K(1)$-locally dualisable, since $\compKU(M(G))$ is a finitely generated $\Z_p$-module; see \cite[Thm.~8.6]{hovey_strickland_Ktheories_localisation}.
We in particular obtain an isomorphism
\[
    \KU_p^*(M^\vee(G)) \cong G.
\]

\begin{remark}
If $G$ is concentrated in degree $d$, then Bousfield \cite{Bous_v1_Hspaces} uses the notation $\mathscr{M}(G,d)$ for $M^\vee(G)$.
\end{remark}

\section{Construction of the presentation}
\label{sec:construction_presentation}

In this section we construct the presentation described in Theorem~\ref{thm:cochain_presentation}.
We begin by specifying the spaces this theorem pertains to.

\begin{definition}
\label{def:exterior_theta_alg}
Let $G$ be a ($p$-adic) Morava module concentrated in odd degree, and let $\theta^p_G \colon G \to G$ be a morphism of Morava modules.
The \defi{exterior $\theta$-algebra} on $(G,\theta^p_G)$ is the $\Z_p$-algebra $\Lambda_{\Z_p}[G]$, with Adams operations and the map $\theta^p$ induced from the Adams operations on $G$ and from $\theta^p_G$, respectively.
We shall often write $\Lambda[G]$ for $\Lambda_{\Z_p}[G]$.
\end{definition}

\begin{center}
    \textit{Throughout this section, we follow Notation~{\upshape\ref{not:ext_algebra}}.}
\end{center}

\begin{notation}
\label{not:ext_algebra}
Let $X$ be a pointed space, and let $(G,\theta^p_G)$ be a pair of a $p$-adic Morava module with an endomorphism, such that
\begin{condenum}
    \item $G$ is finitely generated and free as a $\Z_p$-module, and concentrated in odd degree;
    \item $\KU_p^*(X)$ is an exterior $\theta$-algebra on $(G,\theta^p_G)$, or more precisely, we fix an isomorphism
    \[
        \KU_p^*(X) \cong \Lambda[G]
    \]
    of $p$-adic $\theta$-algebras with Adams operations.
\end{condenum}
Briefly put, the K-theory of $X$ is an exterior algebra on a finite number of odd generators, where the generators have specified Adams operations and $\theta^p$-action.
As previously noted in the Introduction, many H-spaces provide examples: see \cite[Thm.~6.3]{Bous_v1_Hspaces}.
\end{notation}

\begin{definition}
\label{def:theta_minus_Ftheta}
Define a map $\theta^p - F\theta^p_G \colon FG \to FG$ by
\[
    (x_0,x_1,x_2,\dotsc) \mapsto (0,x_0,x_1,\dotsc) - (\theta^p_G(x_0),\ \theta^p_G(x_1),\ \theta^p_G(x_2),\ \dotsc).
\]
Here $\theta^p$ denotes the formal $\theta$-operation in the target (i.e., shifting), and $F\theta^p_G$ denotes the component-wise action of $\theta^p_G$ on $FG = G\oplus G \oplus \dotsb$.
By applying the functor $\mathop{L_0}\Lambda[-]$, we get a map $\Free_\theta[G] \to \Free_\theta[G]$; this map we will also denote by $\theta^p - F\theta^p_G$.
\end{definition}

We also have a map $\Free_\theta[G] \to \Lambda[G]$ induced by the inclusion of the generators $G \to \Lambda[G]$.

\begin{lemma}
\label{lem:pushout_LG_free_theta_alg}
The square
\[
    \begin{tikzcd}
        \Free_\theta[G] \arrow[r,"\theta^p - F\theta^p_G"] \arrow[d] &[1.5em] \Free_\theta[G] \arrow[d]\\
        \Z_p\arrow[r] & \Lambda[G]
    \end{tikzcd}
\]
is a pushout square in $\CAlgh_{\Z_p}^*$.
\end{lemma}

In other words, this is a presentation (or `cokernel sequence')
\[
    \begin{tikzcd}
        \Free_\theta[G] \arrow[r,"\theta^p -F\theta^p_G"] &[1.5em] \Free_\theta[G] \arrow[r] & \Lambda[G]
    \end{tikzcd}
\]
of augmented derived $p$-complete algebras.

\begin{proof}
We have a pushout square
\begin{equation}
    \label{eq:pushout_G_FG}
    \begin{tikzcd}
        FG \arrow[r,"\theta^p - F\theta^p_G"] \arrow[d] &[1.5em] FG \arrow[d]\\
        0\arrow[r] & G \arrow[ul, phantom, "\ulcorner", very near start]
    \end{tikzcd}
\end{equation}
in $\Mod_{\Z_p}^*$.
The free derived $p$-complete algebra functor $\Mod_{\Z_p}^* \to \CAlgh^*_{\Z_p}$ is a left adjoint and therefore preserves colimits.
The three nonzero terms in the square~\eqref{eq:pushout_G_FG} are concentrated in odd degree, so on these objects this functor is given by $\mathop{L_0} \Lambda[-]$.
Applying this functor to the square~\eqref{eq:pushout_G_FG} yields the claimed square in $\CAlgh_{\Z_p}^*$, which is therefore a pushout in $\CAlgh_{\Z_p}^*$.
\end{proof}

The advertised presentation of $\KU_p^{X_+}$
\[
    \begin{tikzcd}
        \KU_p\hotimes \Sym(M(G)) \arrow[r,"R"] & \KU_p\hotimes \Sym(M(G)) \arrow[r,"T"] & \KU_p^{X_+}
    \end{tikzcd}
\]
will be a topological analogue of the above algebraic presentation, with $R$ playing the role of the map $\theta^p - F\theta^p_G$, and $T$ playing the role of the map $\Free_\theta[G] \to \Lambda[G]$.
To realise this, we need some results.

\begin{proposition}[\cite{barthel_heard_E2_localAdams}, Prop.~1.14]
\label{prop:barthel-heard}
Let $Y$ and $Z$ be spectra.
Suppose that $\compKU(Y)$ is \emph{pro-free}, i.e., it is of the form $L_0F$ with $F$ a free $(\KU_p)_*$-module.
Then the natural map
\[
    \mathop{\pi_*} \map\br*{\, Y,\ \KU_p\hotimes Z\, } \to \Hom_{(\KU_p)_*}\br*{\compKU(Y),\ \compKU(Z)}
\]
is an isomorphism.
\end{proposition}
Since $G$ is finitely generated and free, it is immediate that $\Lambda[G]$ is pro-free.
Using Example~\ref{ex:free_theta_alg_odd_generators}, we see that $\Free_\theta[G] = \mathop{L_0}\Lambda[FG]$ is also pro-free.

\begin{corollary}
\label{cor:Kcohom_from_Khom}
Let $Y$ be a spectrum such that $\compKU(Y)$ is pro-free.
Then we have an isomorphism
\[
    \KU_p^*(Y) \cong \Hom_{(\KU_p)_*}\br*{\compKU(Y),\ (\KU_p)_*}.
\]
\end{corollary}
\begin{proof}
Take $Z = \S$ in Proposition~\ref{prop:barthel-heard}.
\end{proof}

\begin{lemma}
\label{lem:Ktheory_smash_product_and_tensor_product}
Let $Y$ and $Z$ be spectra.
Suppose that $\KU_p^*(Y)$ is finitely generated and free.
% as a $\Z_p$-module.
Then the natural map
\[
    \KU_p^*(Y) \otimes_{(\KU_p)^*} \KU_p^*(Z) \to \KU_p^*(Y\otimes Z)
\]
is an isomorphism.
\end{lemma}
\begin{proof}
Because $\KU_p^*(Y)$ is free, the K\"unneth spectral sequence for $\KU_p^*(Y\otimes Z)$ (as in, e.g., \cite[Ch.~IV, Thm.~4.7]{EKMM}) collapses, yielding the claim.
\end{proof}

\begin{construction}
\label{constr:R_and_T}
We now construct two maps 
\begin{align*}
    T \colon &\KU_p\hotimes \Sym(M(G)) \to \KU_p^{X_+},\\
    R \colon &\KU_p\hotimes \Sym(M(G)) \to \KU_p\hotimes \Sym(M(G))
\end{align*}
of $\K(1)$-local augmented $\KU_p$-algebras as follows.
\begin{letterenum}
    \item Because $\KU_p^{X_+}$ is $\K(1)$-local, the map $T$ corresponds to a map of spectra
    \[
        M(G) \to \KU_p^{X_+}.
    \]
    By adjunction, this is the same as a map of spectra $M(G) \otimes \Sigma^\infty_+ X \to \KU_p$, i.e., a cohomology class in $\KU_p^*(M(G)\otimes \Sigma^\infty_+X)$.
    The module $\KU_p^*(M(G))\cong G$ is free, so Lemma~\ref{lem:Ktheory_smash_product_and_tensor_product} applies, yielding an isomorphism
    \[
        \KU_p^*(M(G)\otimes\Sigma^\infty_+ X) \cong \KU_p^*(M(G)) \otimes_{(\KU_p)^*} \KU_p^*(X).
    \]
    Using Corollary~\ref{cor:Kcohom_from_Khom} and the fact that $G$ is finitely generated and free, the right-hand side is isomorphic to
    \[
        \Hom_{\Z_p}(G,\Z_p)\otimes_{\Z_p} \Lambda[G] \cong \Hom_{\Z_p}(G, \Lambda[G]).
    \]
    In conclusion, we have to specify a homomorphism $G \to \Lambda[G]$.
    This we take to be the inclusion of the generators.
    \item The map $R$ corresponds to a map of $\K(1)$-local $\KU_p$-modules
    \[
        \KU_p\hotimes M(G) \to \KU_p\hotimes \Sym(M(G)).
    \]
    By Proposition~\ref{prop:barthel-heard}, such a map is determined up to equivalence by the homomorphism it induces on homotopy groups.
    By Theorem~\ref{thm:mcclure} and Lemma~\ref{lem:K_Moore_spectrum}, the K-homology of $\Sym(M(G))$ is
    \[
        \Free_\theta\sqbr*{\compKU(M(G))} \cong \Free_\theta[G].
    \]
    In conclusion, we have to specify a homomorphism $G \to \Free_\theta[G]$.
    This we take to be the map $\theta^p - F\theta_G^p$ of Definition~\ref{def:theta_minus_Ftheta} (restricted to $G\subseteq \Free_\theta[G]$).
\end{letterenum}
\end{construction}

In a general pointed $\infty$-category, recall that a map is called \emph{nullhomotopic} if it factors over a zero object.
In augmented $\KU_p$-algebras, the algebra $\KU_p$ is a zero object.
\begin{lemma}
\label{lem:composite_null_presentation_Kalg}
The composite $T\circ R$ is a nullhomotopic map of augmented $\KU_p$-algebra spectra.
\end{lemma}
\begin{proof}
By the same reasoning as above, the composite $T\circ R$ is classified by the corresponding map of spectra
\[
    M(G) \to \KU_p^{X_+},
\]
which is classified by the corresponding cohomology class in
\[
    \KU_p^*(M(G)) \otimes_{(\KU_p)^*} \KU_p^*(X) \cong \Hom_{\Z_p}(G, \Lambda[G]).
\]
Unwinding our constructions, we see that this map $G \to \Lambda[G]$ is given by
\[
    \begin{tikzcd}[row sep=tiny]
        G \arrow[r,"\theta^p-F\theta^p_G"] &[1.5em] \Free_\theta[G] \arrow[r] & \Lambda[G],\\
        x \arrow[r,mapsto] & (-\theta_G^p (x),\ x,\ 0,\ 0,\ \dotsc) \arrow[r,mapsto] & - \theta^p_G(x) + \theta^p_G(x) = 0.
    \end{tikzcd}
\]
As this is the zero map, we are done.
\end{proof}

Henceforth we fix a choice of nullhomotopy, which is a witness to the commutativity of the diagram
\begin{equation}
    \label{eq:triangle_for_presentation_KUX}
    \begin{tikzcd}
        \KU_p\hotimes \Sym(M(G)) \arrow[r, "R"] \arrow[d] &[1.75em] \KU_p\hotimes \Sym(M(G)) \arrow[d,"T"]\\
        \KU_p \arrow[r] & \KU_p^{X_+}
    \end{tikzcd}
\end{equation}
in the $\infty$-category $\mathop{L_{\K(1)}} \CAlg^\aug_{\KU_p}$ of augmented $\K(1)$-local $\KU_p$-algebras.

We will prove Theorem~\ref{thm:cochain_presentation} by showing that Diagram~\eqref{eq:triangle_for_presentation_KUX} is a pushout square.
To show this, we will use the $\K(1)$-local Tor spectral sequence constructed in Appendix~\ref{app:K1_Tor_SS}.
Its $E^2$-page consists of the \emph{completed Tor groups} from §\ref{ssec:completed_Tor_groups}:
\[
    E^2_{s,t} = \Torh^{\pi_*A}_{s,t}(\pi_*M,\, \pi_*N),
\]
and in suitable cases it converges to $\pi_{s+t}(M\hotimes_A N)$; see Theorem~\ref{prop:K1_local_Tor_SS}.
We first compute the $E^2$-page of this spectral sequence.

Recall the map $\theta^p - F\theta^p_G \colon FG \to FG$ from Definition~\ref{def:theta_minus_Ftheta}.

\begin{lemma}
\label{lem:theta_minus_Ftheta_is_flat}
The map $\Lambda[\theta^p - F\theta^p_G] \colon \Lambda[FG] \to \Lambda[FG]$ is a flat map of algebras.
\end{lemma}
\begin{proof}
In this proof, we will write $\theta$ for the operation $\theta^p$ on the free $\theta$-algebra, and write $\theta^t$ for the $t$-th iterate of $\theta$, with $\theta^0$ denoting the identity.
Let $x_1,\dotsc,x_n$ be generators for $G$.
We can write
\[
    \Lambda[FG] \cong \Lambda_{\Z_p} [\,x_i,\ \theta x_i,\ \theta^2 x_i,\ \theta^3 x_i, \ \dotsc\,]_{i=1,\dotsc,n}.
\]
For $i=1,\dotsc,n$, write $\theta^p_G(x_i) = \sum_{j=1}^n \lambda_{ij} x_j$ for $\lambda_{ij} \in \Z_p$.
In these coordinates, the map under inspection takes the form
\begin{align*}
    \Lambda_{\Z_p} [\,x_i,\ \theta x_i,\ \theta^2 x_i,\ \dotsc\,]_i &\to \Lambda_{\Z_p} [\,x_i,\ \theta x_i,\ \theta^2 x_i,\ \dotsc\,]_i,\\
    \theta^t x_i &\mapsto \theta^{t+1} x_i - \sum_{j=1}^n \lambda_{ij} \cdot \theta^t x_j.
\end{align*}
Denote the image of $\theta^t x_i$ by $\alpha_{it}$.
The elements
\[
    \set{\alpha_{it} \mid i=1,\dotsc,n, \text{ and } t \geq 0} \cup \set{ x_j \mid j=1,\dotsc,n}
\]
% $\alpha_{it}$ for $i=1,\dotsc,n$ and $t \geq 0$, together with the elements $x_j$ for $j=1,\dotsc,n$,
also serve as exterior algebra generators of $\Lambda[FG]$.
Indeed, the reader can check that the $\alpha_{it}$ anticommute with each other and with the $x_j$, and that the original generators $\set{\theta^t x_i}_{i,t}$ can be expressed as a linear combination of the $\alpha_{it}$ and the $x_j$.
Thus, up to isomorphism, the map $\Lambda[\theta^p - F\theta^p_G] \colon \Lambda[FG] \to \Lambda[FG]$ is a shifting of generators, and is thus a flat map.
\end{proof}

Applying $L_0$ to the map of Lemma~\ref{lem:theta_minus_Ftheta_is_flat} yields a map $\Free_\theta[G] \to \Free_\theta[G]$ of $\theta$-algebras, letting us view the target as a module over the source.
We now compute the following completed Tor groups.

\begin{corollary}
\label{cor:Tor_for_free_theta_vanish}
The completed Tor groups
\[
    \Torh_i^{\Free_\theta[G]}(\Free_\theta[G],\,\Z_p)
\]
vanish for $i > 0$.
\end{corollary}

\begin{proof}
By Lemma~\ref{lem:theta_minus_Ftheta_is_flat}, the Tor groups
\[
    \Tor_i^{\Lambda[FG]}(\Lambda[FG],\,\Z_p)
\]
vanish for $i > 0$.
Now apply \cref{prop:SES_complete_Tor}; note that $L_1(\Lambda[FG])=0$ as the ring is $p$-torsion free.
\end{proof}

\begin{remark}
    \label{rmk:why_use_decompletions}
    Note that in the above computation, it is crucial to work with the de-completion $\Lambda[FG]$ of $\Free_\theta[G]$: it is for these that we can compute the non-completed Tor groups, and \cref{prop:SES_complete_Tor} allows us to deduce the completed case from this.
    By contrast, working with the non-completed Tor spectral sequence would require a computation of the non-completed $\Tor^{\Free_\theta[G]}(\Free_\theta[G],\,\Z_p)$, which we do not know how to do.
\end{remark}

\begin{proof}[Proof of Theorem~{\upshape\ref{thm:cochain_presentation}}]
We show that Diagram~\eqref{eq:triangle_for_presentation_KUX} is a pushout square in the $\infty$-category $\mathop{L_{\K(1)}} \CAlg^\aug_{\KU_p}$.
By \cite[Prop.~1.2.13.8]{HTT}, the pushout of Diagram~\eqref{eq:triangle_for_presentation_KUX} in $\mathop{L_{\K(1)}}\CAlg^\aug_{\KU_p}$ is the same as the pushout in $\mathop{L_{\K(1)}}\CAlg_{\KU_p}$; let $A$ denote this pushout.
The diagram induces a map $A \to \smash{\KU_p^{X_+}}$; we prove this is an equivalence.
It suffices to show that it induces an isomorphism on homotopy groups.
The algebra $A$ is the relative smash product
\[
    A = \Big(\KU_p \hotimes \Sym(M(G)) \Big) \hotimes_{ \KU_p\hotimes \Sym(M(G))} \KU_p.
\]
The homotopy groups of this relative smash product are computed by the $\K(1)$-local Tor spectral sequence of Theorem~\ref{prop:K1_local_Tor_SS}.
Indeed, by Corollary~\ref{cor:Tor_for_free_theta_vanish}, this spectral sequence collapses on the $E^2$-page and is concentrated on the zeroth vertical line.
This in particular shows that it satisfies the convergence criterion of Theorem~\ref{prop:K1_local_Tor_SS}.
Because the spectral sequence is concentrated on the zeroth vertical line, we see that the edge map
\[
    \Free_\theta[G] \hotimes_{\Free_\theta[G]} \Z_p \to \pi_* A
\]
is an isomorphism.
In other words, after taking homotopy groups, the pushout square defining $A$ becomes a pushout square
\[
    \begin{tikzcd}
        \Free_\theta[G] \arrow[r,"\theta^p - F\theta^p_G"] \arrow[d] &[1.5em] \Free_\theta[G] \arrow[d]\\
        \Z_p\arrow[r] & \pi_* A \arrow[ul, phantom, "\ulcorner", very near start]
    \end{tikzcd}
\]
in $\CAlgh_{\Z_p}^*$.

Diagram~\eqref{eq:triangle_for_presentation_KUX} becomes, after taking homotopy groups, the diagram
\[
    \begin{tikzcd}
        \Free_\theta[G] \arrow[r,"\theta^p - F\theta^p_G"] \arrow[d] &[1.5em] \Free_\theta[G] \arrow[d]\\
        \Z_p\arrow[r] & \Lambda[G] \arrow[ul, phantom, "\ulcorner", very near start]
    \end{tikzcd}
\]
from Lemma~\ref{lem:pushout_LG_free_theta_alg}, which we know is a pushout square in $\CAlgh_{\Z_p}^*$.
Thus the map $A \to \KU_p^{X_+}$ induces an isomorphism on homotopy groups, which means it is an equivalence.
\end{proof}

\section{Evaluation of the comparison map}
\label{sec:comparison_map}
Let $X$ be a pointed space as in Notation~\ref{not:ext_algebra}.
The goal of this section is to show that, under some additional conditions on $X$ (see Notation~\ref{not:ext_algebra_plus_Bousfield_stuff}), the Behrens--Rezk comparison map $c_X$ is an equivalence.
Since $\mathop{L_{\K(1)}} \Sigma^\infty_+ X$ is $\K(1)$-locally dualisable (see Lemma~\ref{lem:X_is_suitably_finite}), this implies that $X$ is $\Phi_1$-good.

In §\ref{ssec:comparison_map} we review the construction of the Behrens--Rezk comparison map $c_X$.
Topological Andr\'e--Quillen cohomology features in this comparison map, and for this reason we very briefly review it in §\ref{ssec:TAQ}.
The proof that the comparison map is an equivalence works by constructing the diagram advertised in Theorem~\ref{thm:diagram_of_fibre_seqs}; this is done in §\ref{ssec:construction_of_diagram}.
For this construction we require some computational results by Bousfield, which we summarise in §\ref{ssec:review_Bousfield_v1}.

\subsection{Topological Andr\'e--Quillen cohomology}
\label{ssec:TAQ}

\begin{theorem}
Let $\calC$ be a symmetric monoidal $\infty$-category.
Suppose that $\calC$ has all small colimits and that the tensor product functor preserves these in each variable separately.
Then the trivial algebra functor $\calC \to \CAlg^\aug(\calC)$, $Y \mapsto \mathbf{1}\oplus Y$ has a left adjoint $\TAQ_\calC$.
Moreover, we have a natural equivalence {\textup(}where $Y \in \calC${\textup)}
\[
    \TAQ_\calC(\Sym_\calC(Y)) \simeq Y.
\]
\end{theorem}
\begin{proof}
See, e.g., \cite[§4.2]{gaitsgory_lurie_tamagawa_numbers}.
\end{proof}

We call $\TAQ_\calC$ the \defi{topological Andr\'e--Quillen homology} functor of $\calC$, or \emph{TAQ-homology} for short.
Note that this goes by many different names in the literature, such as the \emph{cotangent fibre} or the \emph{derived indecomposables}.

\begin{example}
Let $A$ be an $\E_\infty$-ring spectrum.
Then we denote by $\TAQ_A$ the TAQ-homology functor of the $\infty$-category $\Mod_A$.
The $A$-linear dual $\TAQ_A(-)^\vee$ we refer to as \defi{topological Andr\'e--Quillen cohomology}.
\end{example}

\begin{example}
Let $E$ be a spectrum, and let $A$ be an $E$-local $\E_\infty$-ring spectrum.
Then the TAQ-homology functor of $\calC=\mathop{L_E}\Mod_A$ is equivalent to the composite $\mathop{L_E}\TAQ_A$.
Indeed, the trivial algebra functor $\mathop{L_E}\Mod_A \to \mathop{L_E}\CAlg^\aug_A$ is given by $M \mapsto A \oplus M$, as the direct sum $A \oplus M$ is $E$-local.
Thus the composite $\mathop{L_E}\TAQ_A$ is left adjoint to the trivial algebra functor of $\calC = \mathop{L_E}\Mod_A$.

The $E$-local version of TAQ-cohomology is the same as the ordinary version: since $A$ is $E$-local, we have an equivalence
\[
    \map_A(\mathop{L_E}\TAQ_A(-),\ A) \simeq \map_A(\TAQ_A(-),\ A).\qedhere
\]
\end{example}

\subsection{The Behrens--Rezk comparison map}
\label{ssec:comparison_map}

\begin{definition}
Let $A$ be a $\K(1)$-local $\E_\infty$-ring spectrum.
In \cite[§6]{Behrens_Rezk_Kn}, Behrens and Rezk construct a natural transformation
\[
    c_A^X\colon \mathop{L_{\K(1)}} \TAQ_A(A^{X_+}) \to A^{\Phi_1 X}
\]
of functors $\Spaces^\op_* \to \mathop{L_{\K(1)}} \Mod_A$.
We call $c_A^X$ the \defi{$A$-linear dual comparison map}.
Taking the $A$-linear dual and precomposing with the natural map
\[
    A \hotimes \Phi_1 X \to \map_A(A^{\Phi_1 X},\ A)
\]
yields the \defi{$A$-linear comparison map}
\[
    c_X^A \colon A \hotimes \Phi_1 X \to \TAQ_A(A^{X_+})^\vee,
\]
which is a natural transformation of functors $\Spaces_* \to \mathop{L_{\K(1)}}\Mod_A$.
We write $c_X$ when $A = \S_{\K(1)}$, and refer to this simply as the \defi{comparison map}.
\end{definition}

\begin{lemma}
\label{lem:Phi_good_iff_K_theoretic_equiv}
Let $X$ be a pointed space such that the spectra $\mathop{L_{\K(1)}}\Sigma^\infty_+ X$ and $\mathop{L_{\K(1)}} \TAQ_{\S_{\K(1)}}(\S_{\K(1)}^{X_+})$ are $\K(1)$-locally dualisable.
Then $c_X$ is an equivalence if and only if the $\KU_p$-linear comparison map
\[
    c_X^{\KU_p} \colon \KU_p \hotimes \Phi_1 X \to \TAQ_{\KU_p}(\KU_p^{X_+})^\vee
\]
is an equivalence.
\end{lemma}
\begin{proof}
The functor $\KU_p \hotimes \blank \colon \Sp_{\K(1)} \to \mathop{L_{\K(1)}}\Mod_{\KU_p}$ is conservative, so $c_X$ is an equivalence if and only if $\KU_p \hotimes c_X$ is an equivalence.
Because of the dualisability assumptions, the natural map
\[
    \KU_p \hotimes \br*{\TAQ_{\S_{\K(1)}}(\S_{\K(1)}^{X_+})^\vee} \to \TAQ_{\KU_p}(\KU_p^{X_+})^\vee
\]
is an equivalence.
This identifies $\KU_p\hotimes c_X$ with $c_X^{\KU_p}$, proving the claim.
\end{proof}

\subsection{Review of Bousfield's computations}
\label{ssec:review_Bousfield_v1}

\begin{recall}
\label{recall:modification_functor}
In \cite[§3.5]{Bous_v1_Hspaces}, Bousfield constructs a functor $\widetilde{(-)}\colon \Sp_{\K(1)} \to \Sp$, together with a natural transformation to the identity.
These enjoy the following properties:
\begin{condenum}
    \item $\mathop{\pi_n} \widetilde{Y} = 0$ for $n<2$;
    \item the map $\widetilde{Y} \to Y$ induces an isomorphism $\mathop{\pi_n} \widetilde{Y}  \cong \mathop{\pi_n} Y$ for $n > 2$;
    \item the map $\widetilde{Y} \to Y$ is a $\K(1)$-localisation.
\end{condenum}
In particular, combining item~(iii) with the natural equivalence $\Phi_1 \Omega^\infty \simeq L_{\K(1)}$ (see \cite[§6]{Behrens_Rezk_Kn} or \cite[Thm.~1.1]{kuhn_guide_telescopic}), we have for $Y \in \Sp_{\K(1)}$ a natural equivalence
\begin{equation}
    \Phi_1 \Omega^\infty \, \widetilde{Y} \simeq Y.\label{eq:Phi_of_modification}
\end{equation}
\end{recall}

\begin{definition}
If $G$ is a Morava module as in §\ref{ssec:Ktheory_Moore_spectra}, then we write $\tM^\vee(G)$ for the modification of \cref{recall:modification_functor} applied to $M^\vee(G)$.
\end{definition}

Bousfield computed the K-cohomology of $\Omega^\infty \tM^\vee(G)$.
This turns out to be a type of free $\theta$-algebra on $G$, but of a slightly different kind than the free $\theta$-algebras introduced in §\ref{ssec:theta_algebras}.
Let $\overline{F}G$ denote $G \times G \times \dotsb$.
If $N$ is a $\Z_p$-module, let $\hLambda[N]$ denote
\[
    \hLambda[N] :=\lim_\alpha \Lambda_{\Z_p}[N_\alpha],
\]
where the $N_\alpha$ are the finite quotients of $N$.
The algebra $\hLambda[\overline{F}G]$ gets the structure of a $p$-adic $\theta$-algebra with Adams operations in a similar way as $\Free_\theta[G]$ from Example~\ref{ex:free_theta_alg_odd_generators}.

Note that for any $\K(1)$-local spectrum $Y$, by item~(iii) of Recollection~\ref{recall:modification_functor}, we have an equivalence
\[
    \KU_p^Y \simeq \KU_p^{\widetilde{Y}}.
\]
In particular, the counit $\Sigma_+^\infty\Omega^\infty \widetilde{Y} \to \widetilde{Y}$ induces a map
\begin{equation}
    \label{eq:counit_on_KUp}
    \KU_p^Y \to \KU_p^{\Omega^\infty \widetilde{Y}_+}.
\end{equation}

\begin{theorem}[\cite{Bous_Ktheory_infloop}; \cite{Bous_v1_Hspaces}, Thm.~3.7]
\label{thm:Ktheory_infloop}
Let $G$ be a Morava module concentrated in odd degree, which is finitely generated and free as $\Z_p$-module.
Then we have an isomorphism of $\theta$-algebras
\[
    \KU_p^*\br*{\Omega^\infty \tM^\vee(G)} \cong \hLambda[\overline{F} G].
\]
Moreover, under this isomorphism, the counit $\Sigma^\infty_+ \Omega^\infty \tM^\vee(G) \to \tM^\vee(G)$ induces the inclusion of generators $G \to \hLambda[\overline{F} G]$ on $\KU_p^*$.
\end{theorem}

Bousfield then used these infinite loop spaces to compute $\Phi_1 X$ for certain $X$.
In the following theorem, we write $\theta^p - \overline{F}\theta^p_G \colon \hLambda[\overline{F}G] \to \hLambda[\overline{F}G]$ for the map defined in the same manner as $\theta^p - F\theta^p_G$ in Definition~\ref{def:theta_minus_Ftheta}, replacing $FG$ by $\overline{F}G$ and $\mathop{L_0}\Lambda[-]$ by $\hLambda[-]$ in the construction.

\begin{theorem}[\cite{Bous_v1_Hspaces}, Thm.~4.8, Thm.~8.1]
\label{thm:Phi_fibre_seq_bousfield}
Let $X$ be a pointed space as in Notation~{\upshape\ref{not:ext_algebra_plus_Bousfield_stuff}}.
Then there exists a commutative diagram of pointed spaces of the form
\[
    \begin{tikzcd}
        \mathop{L_{\K(1)}} X \arrow[r,"h"] \arrow[d] & \Omega^\infty \tM^\vee(G) \arrow[d,"f"]\\
        * \arrow[r] & \Omega^\infty \tM^\vee(G)
    \end{tikzcd}
\]
with the following properties.
\begin{condenum}
    \item After application of $\KU_p^*$, the map $h$ becomes the natural quotient $\hLambda[\overline{F}G] \to \Lambda[G]$.
    \item After application of $\KU_p^*$, the map $f$ becomes $\theta^p - \overline{F}\theta^p_G$; cf.\ Definition~{\upshape\ref{def:theta_minus_Ftheta}}.
    \item After application of $\Phi_1$, the square becomes a pullback square.
    Using the natural identification~\eqref{eq:Phi_of_modification}, this is a fibre sequence of $\K(1)$-local spectra of the form
    \[
        \Phi_1 X \to M^\vee(G) \to M^\vee(G).
    \]
\end{condenum}
\end{theorem}

The maps $f$ and $h$ in the above theorem are in general not unique up to homotopy.
We shall henceforth fix a choice of such a diagram for every such $X$.

\begin{remark}
If $\theta^p_G$ is injective, then the fibre sequence of Theorem~\ref{thm:Phi_fibre_seq_bousfield}\,(iii) simplifies to an equivalence $\Phi_1 X \simeq M^\vee(G/\theta^p_G)$.
\end{remark}

\begin{remark}
\label{rmk:mistake_bousfield99_thm48}
There seems to be a missing assumption in Theorem~4.8 of \cite{Bous_v1_Hspaces}.
The theorem ends with the conclusion that $\uH^1(X;\Z_p)$ and $\uH^2(X;\Z_p)$ vanish under the assumptions posed there, but $X = S^1$ is a counterexample to this.
Because the vanishing of $\uH^1(X;\Z_p)$ and $\uH^2(X;\Z_p)$ is necessary for later results of \cite{Bous_v1_Hspaces} (in particular for Theorem~8.1) we have included it as a separate assumption in Notation~\ref{not:ext_algebra_plus_Bousfield_stuff} below.
Requiring $G$ to be \emph{reduced} in the sense of Definition~2.8 of op.\ cit.\ would imply $\uH^1(X;\Z_p) = 0$ (see also §2.9 of op.\ cit.).
\end{remark}

\subsection{Construction of the diagram}
\label{ssec:construction_of_diagram}

\begin{center}
    \textit{Throughout this section, we follow Notation~{\upshape\ref{not:ext_algebra_plus_Bousfield_stuff}}.}
\end{center}

\begin{notation}
\label{not:ext_algebra_plus_Bousfield_stuff}
Let $X$ be a pointed space as in Notation~\ref{not:ext_algebra} that also satisfies the conditions of \cite[Thm.~8.1]{Bous_v1_Hspaces}.
Concretely, this means that we fix a pair $(G,\theta^p_G)$ of a Morava module with an endomorphism, such that
\begin{condenum}
    \item $G$ is finitely generated and free as a $\Z_p$-module, and concentrated in odd degree;
    \item $\KU_p^*(X)$ is an exterior $\theta$-algebra on $(G,\theta^p_G)$, or more precisely, we fix an isomorphism $\KU_p^*(X) \cong \Lambda[G]$ of $\theta$-algebras, cf.\ Definition~\ref{def:exterior_theta_alg};
    \item the pair $(G,\theta^p_G)$ is \emph{regular} in the sense of \cite[Def.~4.4]{Bous_v1_Hspaces};
    \item the space $X$ is connected and the groups $\uH^1(X;\Z_p)$ and $\uH^2(X;\Z_p)$ vanish;
    \item the space $X$ is \emph{$\K(1)$-durable}, i.e., the map $X \to \mathop{L_{\K(1)}}X$ induces an isomorphism on $v_1$-periodic homotopy groups, i.e., it induces an equivalence $\Phi_1X \simeq \Phi_1(L_{\K(1)}X)$.
\end{condenum}
\end{notation}

\begin{remark}
Bousfield shows that for many H-spaces, the pair $(G,\theta^p_G)$ is regular; see \cite[Lem.~6.1]{Bous_v1_Hspaces}.
He also shows that H-spaces are $\K(1)$-durable; see \cite[Cor.~7.8]{Bous_v1_Hspaces}.
In particular, all simply connected compact Lie groups satisfy the above conditions.
More generally, if $X$ is an H-space such that $\uH_*(X;\Q)$ is associative and $\uH_*(X;\Z_{(p)})$ is finitely generated over $\Z_{(p)}$, then $X$ satisfies the conditions of Notation~\ref{not:ext_algebra_plus_Bousfield_stuff}; see \cite[Thm.~6.3]{Bous_v1_Hspaces}.
\end{remark}

Our goal is to show that $c_X$ is an equivalence for such $X$.
First we show that we can equivalently check this on the $\KU_p$-linear comparison map.
It turns out that this does not require the additional assumptions from Notation~\ref{not:ext_algebra_plus_Bousfield_stuff}: those from Notation~\ref{not:ext_algebra} suffice.

\begin{lemma}
\label{lem:X_is_suitably_finite}
Let $X$ be a pointed space as in Notation~{\upshape\ref{not:ext_algebra}}.
Then $X$ satisfies the conditions of Lemma~{\upshape\ref{lem:Phi_good_iff_K_theoretic_equiv}}, i.e., $\mathop{L_{\K(1)}}\Sigma^\infty_+ X$ and $\mathop{L_{\K(1)}} \TAQ_{\S_{\K(1)}}(\S_{\K(1)}^{X_+})$ are $\K(1)$-locally dualisable.
\end{lemma}
\begin{proof}
By \cite[Thm.~8.6]{hovey_strickland_Ktheories_localisation}, it suffices to show that the $\Z_p$-modules.
\[
    \KU_p^*(X) \qquad \text{and} \qquad \compKU\br*{\TAQ_{\S_{\K(1)}}(\S_{\K(1)}^{X_+})}
\]
are finitely generated.
The first of these is finitely generated by assumption.
Applying $\K(1)$-local TAQ-homology over $\KU_p$ to the cofibre sequence of Theorem~\ref{thm:cochain_presentation}, we obtain a cofibre sequence
\[
    \KU_p \hotimes M(G) \to \KU_p\hotimes M(G) \to \mathop{L_{\K(1)}}\TAQ_{\KU_p}(\KU_p^{X_+})
\]
of $\K(1)$-local $\KU_p$-modules.
Since $\pi_*(\KU_p\hotimes M(G)) \cong G$ is finitely generated, this shows that
\[
    \pi_*\br*{\mathop{L_{\K(1)}}\TAQ_{\KU_p}(\KU_p^{X_+})}
\]
is also finitely generated.
Because we have an equivalence
\[
    \mathop{L_{\K(1)}}\TAQ_{\KU_p}(\KU_p^{X_+}) \simeq \KU_p\hotimes \mathop{L_{\K(1)}} \TAQ_{\S_{\K(1)}}(\S_{\K(1)}^{X_+}),
\]
this establishes the claim.
\end{proof}

Before we proceed with the construction of the diagram advertised in Theorem~\ref{thm:diagram_of_fibre_seqs}, we require some notation and a small fact about the comparison map.
The counit $\Sigma^\infty_+ \Omega^\infty Y \to Y$ induces a natural transformation
\begin{equation}
    \label{eq:eps_map}
    \eps_Y \colon \Sym_{\KU_p}(\KU_p^Y) \to \KU_p^{\Omega^\infty Y_+}
\end{equation}
of functors $\Sp \to \CAlg_{\KU_p}^\aug$.
This in turn induces a natural transformation
\[
    \TAQ(\eps)_Y \colon \TAQ_{\KU_p}(\Sym_{\KU_p}(\KU_p^Y)) \to \TAQ_{\KU_p}(\KU_p^{\Omega^\infty Y_+})
\]
of functors $\Sp \to \Mod_{\KU_p}$.

The following result was originally proved only when working over $\S_{\K(n)}$, but the proof applies equally well to the $\KU_p$-linear setting.
\begin{lemma}[\cite{Behrens_Rezk_Kn}, Lem.~6.1]
\label{lem:behrens_rezk_id}
If $Y$ is a spectrum, then the composite
\[
    \begin{tikzcd}
        \KU_p^Y \arrow[d,"\simeq"'] &[1.5em] &[1.2em] \KU_p^Y \\
        \TAQ_{\KU_p}(\Sym_{\KU_p}(\KU_p^Y)) \arrow[r, "\TAQ(\eps)_Y"] &[1.5em] \TAQ_{\KU_p}\br*{\KU_p^{\Omega^\infty Y_+}} \arrow[r, "c^{\,\Omega^\infty Y}_{\KU_p}"] &[1.2em] \KU_p^{\Phi_1(\Omega^\infty Y)} \arrow[u,"\simeq"']
    \end{tikzcd}
\]
is the identity.
\end{lemma}

Denote by $\TAQ(\eps)^\vee_Y$ the $\KU_p$-linear dual of $\TAQ(\eps)_Y$.
From the definition of the comparison map, we deduce the following.

\begin{corollary}
\label{cor:left_inv_for_comp_map}
If $Y$ is a $\K(1)$-locally dualisable spectrum, then the composite
\[
    \begin{tikzcd}
        \KU_p \hotimes Y \arrow[d,"\simeq"'] &[1.2em] &[1.5em] \map(Y^\vee, \KU_p) \\
        \KU_p\hotimes \Phi_1 (\Omega^\infty Y) \arrow[r, "c_{\Omega^\infty Y}^{\KU_p}"] & \TAQ_{\KU_p}(\KU_p^{\Omega^\infty Y_+})^\vee \arrow[r, "\TAQ(\eps)^\vee_Y"] & \TAQ_{\KU_p}(\Sym_{\KU_p}(\KU_p^Y))^\vee \arrow[u,"\simeq"']
    \end{tikzcd}
\]
is an equivalence.
\end{corollary}

\begin{construction}
\label{constr:big_diagram_Thm_B}
We now construct a commutative diagram of $\K(1)$-local $\KU_p$-modules of the form
\[
    \begin{tikzcd}
        \KU_p \hotimes \Phi_1 X \arrow[r] \arrow[d, equals] & \KU_p\hotimes M^\vee(G) \arrow[r] \arrow[d,"\simeq"] & \KU_p \hotimes M^\vee(G) \arrow[d,"\simeq"]\\
        \KU_p \hotimes \Phi_1 X \arrow[r] \arrow[d,"c_X^{\KU_p}"] & \KU_p \hotimes \Phi_1\br*{\Omega^\infty \tM^\vee(G)} \arrow[r] \arrow[d,"c^{\KU_p}"] & \KU_p \hotimes \Phi_1\br*{\Omega^\infty \tM^\vee(G)} \arrow[d,"c^{\KU_p}"]\\
        \TAQ_{\KU_p}(\KU_p^{X_+})^\vee \arrow[r,"h"] \arrow[d, equals] & \TAQ_{\KU_p}\br*{\KU_p^{\Omega^\infty \tM^\vee(G)_+}}^\vee \arrow[r,"f"] \arrow[d,"\TAQ(\eps)^\vee"] & \TAQ_{\KU_p}\br*{\KU_p^{\Omega^\infty \tM^\vee(G)_+}}^\vee \arrow[d,"\TAQ(\eps)^\vee"]\\
        \TAQ_{\KU_p}(\KU_p^{X_+})^\vee \arrow[r] & \map(M(G),\ \KU_p) \arrow[r] & \map(M(G),\ \KU_p),
    \end{tikzcd}
\]
as follows.
\begin{letterenum}
    \item The first row is obtained by applying $\KU_p \hotimes {-}$ to the fibre sequence of Theorem~\ref{thm:Phi_fibre_seq_bousfield}.
    \item In the first triple of vertical maps, we have used the natural equivalence~\eqref{eq:Phi_of_modification} to identify $\Phi_1 \Omega^\infty \tM^\vee(G)$ with $M^\vee(G)$.
    % $\Phi_1 \Omega^\infty \simeq L_{\K(1)}$ and $\mathop{L_{\K(1)}}\tM^\vee(G)\simeq M^\vee(G)$.
    \item The second triple of vertical maps is the $\KU_p$-linear Behrens--Rezk comparison map.
    For the left vertical map, it is the comparison map for the space $X$, and for the last two it is the comparison map for the space $\Omega^\infty \tM^\vee(G)$.
    \item The third row is (up to equivalence) obtained by taking the TAQ-cohomology of the $\KU_p$-cochains of the diagram of Theorem~\ref{thm:Phi_fibre_seq_bousfield}.
    \item In the third triple of vertical maps, the last two are given by the natural transformation $\TAQ(\eps)^\vee$ discussed above.
    This defines in particular the fourth row.
\end{letterenum}
\end{construction}

All squares are obtained by applying a natural transformation, and so they come with homotopies witnessing their commutativity.
We still need to identify the bottom row with the fibre sequence
\[
    \begin{tikzcd}
        \TAQ_{\KU_p}(\KU_p^{X_+})^\vee \arrow[r,"\TAQ(T)^\vee"] &[1.5em] \map(M(G),\ \KU_p) \arrow[r,"\TAQ(R)^\vee"] &[1.5em] \map(M(G),\ \KU_p),
    \end{tikzcd}
\]
obtained by taking TAQ-cohomology of the presentation from Theorem~\ref{thm:cochain_presentation}.
This will follow from the following two lemmas.

\begin{lemma}
\label{lem:comm_diagram1}
The diagram of $\K(1)$-local $\KU_p$-modules
\[
    \begin{tikzcd}
        \KU_p\hotimes M(G) \arrow[r,"\widetilde{\eps}"] \arrow[dr,"T"'] & \KU_p^{\Omega^\infty \tM^\vee(G)_+}\arrow[d,"h^*"]\\
        & \KU_p^{X_+}
    \end{tikzcd}
\]
where $\widetilde{\eps}$ denotes the map \eqref{eq:counit_on_KUp} for $Y = \tM^\vee(G)$, commutes up to homotopy.
\end{lemma}
\begin{proof}
Note that we have used the identification
\[
    \KU_p \hotimes M(G) \simeq \KU_p^{M^\vee(G)},
\]
using that $M(G)$ is $\K(1)$-locally dualisable, to identify the source of $\widetilde{\eps}$.
Arguing in the same way as in Construction~\ref{constr:R_and_T}, we see that the map $\widetilde{\eps}$ corresponds to an element
\[
    i \in \KU_p^*\br*{M(G) \otimes \Sigma^\infty_+ \Omega^\infty \tM^\vee(G)} \cong \Hom_{\Z_p}(G,\hLambda[FG]).
\]
By Theorem~\ref{thm:Ktheory_infloop}, the map $i$ is the inclusion of generators $G \to \hLambda[FG]$.
Similarly, the map $h^*\circ \widetilde{\eps}$ corresponds to an element
\[
    j \in \KU_p^* \Big(M(G) \otimes \Sigma^\infty_+ X\Big) \cong \Hom_{\Z_p}(G,\Lambda[G]).
\]
Moreover, the map $j$ is equal to $\pi_*(h^*) \circ i$.
On homotopy groups, the map $h^*$ becomes the natural quotient $\hLambda[FG]\to\Lambda[G]$.
Thus $j$ is equal to the composite $G \to \hLambda[\overline{F}G] \to \Lambda[G]$, or in other words, $j$ is the inclusion of generators $G \to \Lambda[G]$.
Therefore, by definition of $T$ in Construction~\ref{constr:R_and_T}, the composite $h^*\circ \widetilde{\eps}$ is homotopic to $T$.
\end{proof}

In the next lemma, note that for a spectrum $Y$, the target of the map $\eps_Y$ from \eqref{eq:eps_map} is $\K(1)$-local.
It therefore factors over the $\K(1)$-localisation of the source, resulting in a map of the form
\[
    \KU_p \hotimes \Sym(L_{\K(1)} Y^\vee) \to \KU_p^{\Omega^\infty Y_+}.
\]
We will denote this map by $\eps_Y$ as well in the following.

\begin{lemma}
\label{lem:comm_diagram2}
The diagram of $\K(1)$-local $\KU_p$-modules
\[
    \begin{tikzcd}
        \KU_p \hotimes M(G) \arrow[r,"R"] \arrow[d,"\widetilde{\eps}"'] &[1.5em] \KU_p \hotimes \Sym(M(G)) \arrow[d,"\eps"]\\
        \KU_p^{\Omega^\infty \tM^\vee(G)_+} \arrow[r,"f^*"] & \KU_p^{\Omega^\infty \tM^\vee(G)_+}
    \end{tikzcd}
\]
where $\widetilde{\eps}$ denotes the map of the same name from \cref{lem:comm_diagram1} and where $\eps$ denotes the map explained above for $Y= \tM^\vee(G)$, commutes up to homotopy.
\end{lemma}
\begin{proof}
After taking homotopy groups, the diagram becomes
\[
    \begin{tikzcd}
        G \arrow[r,"\theta^p-F\theta^p_G"] \arrow[d] &[1.5em] \Free_\theta[G] \arrow[d]\\
        \hLambda[\overline{F}G] \arrow[r,"\theta^p-\overline{F}\theta^p_G"] &[1.5em] \hLambda[\overline{F}G],
    \end{tikzcd}
\]
with the vertical maps being the natural maps by Theorem~\ref{thm:Ktheory_infloop}.
This diagram obviously commutes, so Proposition~\ref{prop:barthel-heard} implies the original diagram commutes up to homotopy.
\end{proof}

\begin{proof}[Proof of Theorem~{\upshape\ref{thm:diagram_of_fibre_seqs}}]
The diagram from Construction~\ref{constr:big_diagram_Thm_B} has the desired properties.
Indeed, using Lemmas~\ref{lem:comm_diagram1} and \ref{lem:comm_diagram2}, we may identify the bottom row with the fibre sequence
\[
    \begin{tikzcd}
        \TAQ_{\KU_p}(\KU_p^{X_+})^\vee \arrow[r,"\TAQ(T)^\vee"] &[1.5em] \map(M(G),\ \KU_p) \arrow[r,"\TAQ(R)^\vee"] &[1.5em] \map(M(G),\ \KU_p)
    \end{tikzcd}
\]
obtained by taking TAQ-cohomology of the presentation from Theorem~\ref{thm:cochain_presentation}.
By Corollary~\ref{cor:left_inv_for_comp_map}, composing the middle three or the right-most three vertical maps yields an equivalence.
\end{proof}

\appendix

\section{The \texorpdfstring{{\boldmath$\K(1)$}}{K(1)}-local Tor spectral sequence}
\label{app:K1_Tor_SS}

\subfile{subfiles/appendix.tex}

\phantomsection
\printbibliography[heading=bibintoc]

\end{document}

%% file: subfiles/appendix.tex
\begin{center}
    \textit{This appendix is joint work with Max Blans.}% \footnote{\href{mailto:m.a.blans@uu.nl}{\texttt{m.a.blans@uu.nl}}}}
\end{center}

The goal of this appendix is to set up a $\K(1)$-local analogue of the Tor spectral sequence
\[
    E^2_{s,t} = \Tor^{\pi_*A}_{s,t}(\pi_*M,\, \pi_*N) \implies \pi_{s+t}(M\otimes_A N)
\]
where $A$ is an $\E_1$-ring spectrum, $M$ is a right $A$-module spectrum, and $N$ is a left $A$-module spectrum; standard references are \cite[Ch.~IV, Thm.~4.1]{EKMM} and \cite[Prop.~7.2.1.19]{HA}.
In the $\K(1)$-local case, both $A$, $M$ and $N$ are assumed to be $\K(1)$-local, the Tor groups are replaced by Tor groups in the abelian category of derived $p$-complete modules over $\pi_* A$, and the abutment is the homotopy groups of the $\K(1)$-local smash product $M\hotimes_A N = L_{\K(1)} (M\otimes_A N)$.

Let us first recall Lurie's construction of the Tor spectral sequence from \cite[§7.2.1]{HA}. 
This consists of choosing a good simplicial resolution $P_\bullet$ of $N$ by free $A$-modules, and then taking the spectral sequence associated to the simplicial object $M \otimes_A P_\bullet$, as constructed in \cite[§1.2.2, §1.2.4]{HA}.
In setting up the $\K(1)$-local analogue we will proceed in the same way. 
We start by taking $P_\bullet$ to be a good simplicial resolution of $N$ in the $\infty$-category of $\K(1)$-local $A$-modules, and we then take the spectral sequence associated to $M \hotimes_A P_\bullet$.

However, it is not in general true that this spectral sequence converges to $\pi_*(M\hotimes_A N)$.
This failure boils down to the fact that the functor $\pi_* \colon \Sp_{\K(1)} \to \Modh_{\Z_p}^*$ does not preserve sequential colimits.
There are nevertheless some practical conditions under which the spectral sequence does converge.
More precisely, we get the following result.

\begin{theorem}
\label{prop:K1_local_Tor_SS}
Let $A$ be a $\K(1)$-local $\E_1$-ring spectrum, $M$ a $\K(1)$-local right $A$-module spectrum, and $N$ a $\K(1)$-local left $A$-module spectrum.
There exists a spectral sequence $\set{E^r_{s,t},\ d^r}$ of derived $p$-complete $\Z_p$-modules, with $E^2$-page
\[
    E^2_{s,t} = \Torh_{s,t}^{\pi_*A}(\pi_*M, \ \pi_* N),
\]
and with differentials $d^r$ of bidegree $(-r,\ r-1)$.
If on the $E^\infty$-page every anti-diagonal {\upshape(}i.e., all $E^\infty_{s,t}$ with $s+t=n$ for fixed $n${\upshape)} contains only finitely many nonzero terms, then the spectral sequence converges to
\[
    \pi_{s+t}(M\hotimes_A N).
\]
\end{theorem}

\begin{remark}
The condition stated to ensure convergence is not optimal, but is sufficient for our purposes.
\end{remark}

\begin{remark}
This spectral sequence was already considered in \cite[§3.3]{max_blans_masters}, but the subtleties regarding convergence were overlooked there.
The applications considered in \cite{max_blans_masters} remain valid, as they satisfy the above convergence criterion.
\end{remark}

\begin{remark}
    In many cases, the $\Torh$ groups making up the $E^2$-page of the spectral sequence can be computed in terms of ordinary (uncompleted) Tor groups.
    See Proposition~\ref{prop:SES_complete_Tor} for a precise statement.
\end{remark}

\begin{remark}
The same constructions work in the $\K(n)$-local setting for any $n\geq 2$, but in those cases it is not in general possible to identify the $E^2$-page as completed Tor groups.
The reason for this is that the functor $\pi_* \colon \Sp_{\K(n)} \to \Modh_{\Z_p}^*$ preserves coproducts only if $n=1$.
\end{remark}

The rest of the appendix is structured as follows.
In §\ref{ssec:completed_Tor_groups} we define and study the completed Tor groups that make up the $E^2$-page.
In §\ref{ssec:SS_in_K1_local_spectra} we study the spectral sequence associated to a filtered object in spectra in the case that all these spectra are $\K(1)$-local.
This is where the convergence criterion appears.
Finally, the proof of Theorem~\ref{prop:K1_local_Tor_SS} is given in §\ref{ssec:construction_Tor_SS}.

The main text works mainly in the $\KU_p$-linear setting, so that there all homotopy groups can be regarded as $\Z/2$-graded.
Due to the more general setting of this appendix, we are forced to work with $\Z$-graded objects instead.

\subsection{Completed Tor groups}
\label{ssec:completed_Tor_groups}
Unlike in the main text, let $\Modh_{\Z_p}^*$ denote the category of \emph{$\Z$-graded} derived $p$-complete $\Z_p$-modules.
This category is an abelian subcategory of $\Mod_{\Z_p}^*$, closed under $\Ext^*$ and extensions; see \cite[Thm.~A.6]{hovey_strickland_Ktheories_localisation}.
It has all small limits and colimits, with limits being the same as in $\Mod_{\Z_p}^*$, and colimits being computed by taking $L_0$ of the colimit in $\Mod_{\Z_p}^*$; see \cite[§A.2]{barthel_frankland_power_op_MoravaE}.
Note that finite colimits however are the same in both $\Mod_{\Z_p}^*$ and $\Modh_{\Z_p}^*$.
The category $\Modh_{\Z_p}^*$ becomes a symmetric monoidal category under the completed tensor product
\[
    M \hotimes_{\Z_p} N := L_0(M \otimes_{\Z_p} N).
\]
This monoidal product preserves small colimits in each variable separately, and turns $L_0$ into a symmetric monoidal functor $\Mod_{\Z_p}^* \to \Modh_{\Z_p}^*$; see (the proof of) \cite[Cor.~A.7]{hovey_strickland_Ktheories_localisation}.

Let $R$ be an associative algebra object in $\Modh_{\Z_p}^*$, in other words, a $\Z$-graded $\Z_p$-algebra whose underlying $\Z_p$-module is derived $p$-complete.
Let $\LModh_R^*$ denote the category of left module objects over $R$ in $\Modh_{\Z_p}^*$, and similarly let $\RModh_R^*$ denote the right module objects.
These are again abelian categories, and the forgetful functor to $\Modh_{\Z_p}^*$ creates all small limits and colimits; see \cite[Prop.~1.2.14]{martyImmersionsOuvertesMorphismes2009}.
Given $M \in \RModh_R^*$ and $N \in \LModh_R^*$, we can form their relative tensor product via the formula
\[
    M \hotimes_R N := \colim \left ( 
    \begin{tikzcd}[column sep=2em]
        M \hotimes_{\Z_p} R \hotimes_{\Z_p} N \ar[r,shift left=.75ex] \ar[r,shift right=.75ex,swap] & M \hotimes_{\Z_p} N
    \end{tikzcd}
    \right ),
\]
where the equaliser is computed in $\Modh_{\Z_p}^*$ (or equivalently in $\Mod_{\Z_p}^*$), and where the two maps are given by the right and left action of $R$ on $M$ and $N$ respectively.

\begin{lemma}
    \label{lem:L0_relative_tensor}
    Let $R$ be a graded $\Z_p$-algebra, let $M\in \RMod_R^*$, and let $N\in\LMod_R^*$.
    Then the natural map
    \[
        L_0 M\hotimes_{L_0R} L_0 N \to L_0(M\otimes_R N)
    \]
    is an isomorphism.
\end{lemma}
\begin{proof}
    This follows from the definition of $M\otimes_R N$ as the coequaliser
    \[
        \colim \left ( 
        \begin{tikzcd}[column sep=2em]
            M \otimes_{\Z_p} R \otimes_{\Z_p} N \ar[r,shift left=.75ex] \ar[r,shift right=.75ex,swap] & M \otimes_{\Z_p} N
        \end{tikzcd}
        \right )
    \]
    together with the fact that $L_0$ preserves colimits and is symmetric monoidal.
\end{proof}

Going forward, we will make repeated use of the following.

\begin{recall}[\cite{hovey_strickland_Ktheories_localisation}, Thm.~A.2]
    Consider a short exact sequence of graded $\Z_p$-modules
    \[
        0 \to A \to B \to C \to 0.
    \]
    We then obtain an induced long exact sequence
    \[
        0 \to L_1A \to L_1B \to L_1C \to L_0A \to L_0B \to L_0C \to 0
    \]
    where $L_1$ is the first right derived functor of $p$-completion.
\end{recall}

The abelian category $\Modh_{\Z_p}^*$ has enough projectives, and an object is projective if and only if it is of the form $L_0 F$, where $F$ is a free $\Z$-graded $\Z_p$-module; see \cite[§A.4]{hovey_strickland_Ktheories_localisation}.
It immediately follows from this that $\LModh_R^*$ has enough projectives.

\begin{definition}
\label{def:Torh}
Given $M \in \RModh_R^*$, let $\Torh_{i}^R(M, -)$ denote the $i$-th left derived functor of $M \hotimes_R {-} \colon \LModh_R^* \to \Modh_{\Z_p}^*$.
We refer to these as \defi{completed Tor groups}.
\end{definition}

\begin{remark}
    The completed Tor groups above are a relative version (at height $1$) of the completed Tor groups appearing in \cite[§1.1]{baker_Lcomplete_Hopf}.
    (Note that the ring denoted by $R$ in loc.\ cit.\ is in our case given by $\Z_p$, and not by what we call $R$ above.)
\end{remark}

We choose here to resolve the second variable instead of the first.
It turns out that this does not matter, as we now show.
% More precisely, we have the following result.

At various points in this section, we want to allow ourselves to work with $R$ as arising as the derived $p$-completion of another $\Z_p$-algebra $S$.
This is useful in practice; see, e.g., \cref{rmk:why_use_decompletions} in the main text.
% Note that we allow ourselves to work with a ring and modules over it that are not derived $p$-complete; this is useful in practice.

\begin{proposition}
    \label{prop:Tor_biresolution}
    Let $S$ be a graded $\Z_p$-module such that $L_1S$ vanishes.
    Let $M$ be a graded right $S$-module, and $N$ a graded left $S$-module.
    Let $F_\bullet$ be an $S$-free resolution of $M$, and let $G_\bullet$ be an $S$-free resolution of $N$.
    Then $\Torh_*^{L_0S}(L_0M,L_0N)$ can be computed as
    \begin{itemize}
        \item the homology of the complex $L_0(F_\bullet \otimes_S N)$,
        \item the homology of the complex $L_0(M\otimes_S G_\bullet)$,
        \item the homology of $L_0$ applied to the total complex of $F_\bullet \otimes_S G_\bullet$.
    \end{itemize}
\end{proposition}

For the proof, we require the following fact, which is special to the height $1$ case we find ourselves in.
Beware that for higher heights, the second conclusion is false; see \cite[§1.3]{hovey_filtered_colimits_2008}.
In fact, \cref{prop:Tor_biresolution} is false for higher heights as well; see \cite[§1.1, App.~B]{baker_Lcomplete_Hopf} for a discussion.

\begin{lemma}
    \label{lem:L1_sums}
    Let $\set{M_\alpha}$ be a collection of graded $\Z_p$-modules that are all $L_1$-acyclic.
    Then we have
    \[
        L_1 \bigoplus_\alpha M_\alpha = 0.
    \]
    In particular, direct sums in $\Modh_{\Z_p}^*$ are exact.
\end{lemma}
\begin{proof}
    This is \cite[Prop.~1.4]{hovey_filtered_colimits_2008}, although he only records the second statement.
    Let us repeat his proof here for completeness.
    The functor $L_1$ left exact; in fact, it is given by $\Hom(\Z/p^\infty,\blank)$, so it even preserves arbitrary products.
    As we have an inclusion
    \[
        \bigoplus_\alpha M_\alpha \hookto \prod_\alpha M_\alpha,
    \]
    we therefore learn that by applying $L_1$, we obtain an inclusion
    \[
        L_1 \bigoplus_\alpha M_\alpha \hookto \prod_\alpha L_1 M_\alpha,
    \]
    but the target vanishes, showing the first claim.
    The second claim follows from this using that derived $p$-complete modules are $L_1$-acyclic; see \cite[Thm.~A.6]{hovey_strickland_Ktheories_localisation}.
\end{proof}

\begin{proof}[Proof of \cref{prop:Tor_biresolution}]
    In this proof, let us write $\LTorh^{L_0S}_i(\blank,L_0N)$ for the $i$-th derived functor of $\blank \hotimes_{L_0S} L_0N$, and write $\RTorh^{L_0S}_i(L_0M,\blank)$ for the functor as defined in \cref{def:Torh}.

    Since $F_i$ is a sum of shifts of $S$, it follows from \cref{lem:L1_sums} that $L_1 F_i=0$ for all $i$.
    One then easily deduces from the long exact sequence on $L_0$ that the homology of the chain complex $L_0F_\bullet$ vanishes in positive degrees, and is $L_0M$ in degree zero.
    Since $F_i$ is a direct sum of (shifts of) $S$, we find that each $L_0 F_i$ is projective in $\RModh_{L_0S}^*$, so that $L_0F_\bullet$ is a projective resolution of $L_0M$ in this category.
    Likewise, $L_0G_\bullet$ is a projective resolution of $L_0N$ in the category $\LModh_{L_0S}^*$.

    % Next, we note that as $M\otimes_R N$ is defined as the coequaliser
    % \[
    %     \colim \left ( 
    %     \begin{tikzcd}[column sep=2em]
    %         M \otimes_{\Z_p} R \otimes_{\Z_p} N \ar[r,shift left=.75ex] \ar[r,shift right=.75ex,swap] & M \otimes_{\Z_p} N
    %     \end{tikzcd}
    %     \right )
    % \]
    % and $L_0$ preserves colimits and is symmetric monoidal, it follows that
    % % It follows from the definition that we have an isomorphism
    % \[
    %     L_0M \hotimes_{L_0R} L_0N \cong L_0(M \otimes_R N),
    % \]
    % and likewise for $F_i$ in the place for $M$, or $G_j$ in the place for $N$.

    By \cref{lem:L0_relative_tensor}, we have an isomorphism $(L_0F_i) \hotimes_{L_0S}(L_0G_j) \cong L_0(F_i\otimes_S G_j)$.
    As a result, for every fixed $i$ and $j$, the homology of the complexes $L_0(F_\bullet \otimes_S G_j)$ and $L_0(F_i \otimes_S G_\bullet)$ is given, respectively, by
    \[
        \LTorh^{L_0S}_*(L_0M,\,L_0G_j) \qquad \text{and}\qquad \RTorh^{L_0S}_*(L_0F_i,\,L_0N).
    \]
    Note however that the functors $(L_0F_i)\hotimes_S \blank$ and $\blank\hotimes_S (L_0G_j)$ are exact by \cref{lem:L1_sums}, so that both of these vanish in positive degrees, and in degree zero are given, respectively, by
    \[
        (L_0M)\hotimes_{L_0S} (L_0G_j) \qquad \text{and}\qquad (L_0F_i)\hotimes_{L_0S} (L_0N).
    \]
    As a result, the two spectral sequences associated to the double complex $L_0(F_\bullet \otimes_S G_\bullet)$ collapse, and the total homology of $L_0(F_\bullet \otimes_S G_\bullet)$ is computed as the homology of either of the complexes
    \[
        (L_0M)\hotimes_{L_0S} (L_0G_\bullet) \qquad \text{or} \qquad (L_0F_\bullet) \hotimes_{L_0S} (L_0N).
    \]
    The homology of the first is $\RTorh_*^{L_0S}(L_0M,L_0N)$, and the homology of the second is $\LTorh_*^{L_0S}(L_0M,L_0N)$.
    In particular, these two variants of the completed Tor groups are isomorphic.
    Note that the first of these are the groups $\Torh_*^{L_0S}(L_0M,L_0N)$ in the sense of \cref{def:Torh}.
    This finishes the proof.
\end{proof}

For $i=0$, the completed Tor group $\Torh^R_0(M,N)$ is simply given by $L_0(M\otimes_R N)$, or in other words, by $L_0\Tor^R_0(M,N)$.
Our next goal is to compute the higher completed Tor group up to an extension problem; see \cref{prop:SES_complete_Tor} below.

We begin by measuring the failure of $L_0$ to commute with homology.

\begin{lemma}
    \label{lem:reversing_L0_homology}
    Let $A$, $B$ and $C$ be graded $\Z_p$-modules, and let
    \[
        \begin{tikzcd}
            A \rar["f"] & B \rar["g"] & C
        \end{tikzcd}
    \]
    be a sequence such that $gf = 0$.
    Write $H = \ker g / \im f$.
    Suppose that $B$ and $C$ are $L_1$-acyclic.
    Then we have a short exact sequence
    \[
        0 \to L_0 H \to \ker L_0g/\im L_0f \to L_1 H \to 0.
    \]
\end{lemma}
\begin{proof}
    Write $\widetilde{H}$ for $\ker L_0g/\im L_0f$.
    Using right exactness of $L_0$, we then have a commutative diagram with exact rows
    \[
        \begin{tikzcd}
            & L_0\im f \rar \dar["\alpha"] & L_0 \ker g \rar \dar["\beta"] & L_0 H \rar \dar & 0\\
            0 \rar & \im L_0 f \rar & \ker L_0 g \rar & \widetilde{H} \rar & 0
        \end{tikzcd}
    \]
    where the vertical maps are the natural comparison maps (using the right exactness of $L_0$ to define $\alpha$).
    The claim will follow from the snake lemma together with the following two claims: the map $\alpha$ is surjective, and the map $\beta$ is injective with cokernel $L_1 H$.
    We now verify these two claims.

    First, we verify that $L_1 \im f = 0$.
    Observe that we have a short exact sequence
    \begin{equation}
        \label{eq:image_coker_ses}
        0 \to \im f \to B \to \coker f \to 0.
    \end{equation}
    Since $B$ is $L_1$-acyclic and $L_1$ is left exact, this implies that $L_1 \im f = 0$.
    As $C$ is assumed to be $L_1$-acyclic as well, the exact same argument shows that $L_1 \im g=0$ also.

    Next, we show that $\alpha$ is surjective and we compute its kernel.
    Using the short exact sequence~\eqref{eq:image_coker_ses}, we obtain a commutative diagram with exact rows
    \[
        \begin{tikzcd}
            0 \rar & L_1 \coker f \rar & L_0 \im f \dar["\alpha"] \rar & L_0 B \rar \dar[equals] & L_0 \coker f \rar \dar["\cong"] & 0\\
            & 0 \rar & \im L_0 f \rar & L_0 B \rar & \coker L_0 f \rar & 0.
        \end{tikzcd}
    \]
    Here the right vertical map is an isomorphism because $L_0$ is right exact.
    It follows that $\alpha$ is surjective with kernel $L_1 \coker f$.
    To further compute this kernel, note that we have a short exact sequence
    \[
        0 \to H \to \coker f \to \im g \to 0.
    \]
    We previously showed that $L_1 \im g=0$, so applying $L_1$ to this short exact sequence results in an exact sequence
    \[
        0 \to L_1 H \to L_1 \coker f \to 0.
    \]
    In summary, we find that $\alpha$ is surjective with kernel isomorphic to $L_1 H$.
    
    Now consider the short exact sequence
    \[
        0 \to \ker f \to A \to \im f \to 0.
    \]
    Applying $L_0$, we obtain a short exact sequence because $L_1\im f=0$.
    As a result, we have a commutative diagram with short exact rows
    \[
        \begin{tikzcd}
            0 \rar & L_0 \ker f \rar \dar["\beta"] & L_0 A \rar \dar[equals] & L_0\im f \rar \dar["\alpha"] & 0\\
            0 \rar & \ker L_0 f \rar & L_0 A \rar & \im L_0 f \rar & 0.
        \end{tikzcd}
    \]
    Applying the snake lemma and our earlier computation of $\ker \alpha$, we learn that we have a short exact sequence
    \[
        \begin{tikzcd}
            0 \rar & L_0 \ker f \rar["\beta"] & \ker L_0 f \rar & L_1 H \rar & 0.
        \end{tikzcd}\qedhere
    \]
\end{proof}

\begin{proposition}
    \label{prop:SES_complete_Tor}
    Let $S$ be a graded $\Z_p$-algebra such that $L_1 S$ vanishes.
    Let $M$ be a graded right $S$-module, and $N$ a graded left $S$-module.
    Then for every $i>0$, we have a short exact sequence
    \[
        0 \to L_0 \Tor^S_i(M,N) \to \Torh^{L_0S}_i(L_0M, L_0N) \to L_1 \Tor^S_i(M,N) \to 0.
    \]
    In particular, if $\Tor^S_i(M,N)$ vanishes, then $\Torh^{L_0S}_i(L_0M,L_0N)$ also vanishes.
\end{proposition}
\begin{proof}
    Choose $S$-free resolutions $F_\bullet \to M$ and $G_\bullet \to N$.
    Then the homology of the total complex of $F_\bullet \otimes_S G_\bullet$ is given by $\Tor^S_*(M,N)$.
    Using \cref{prop:Tor_biresolution}, we may compute the group $\Torh_i^{L_0S}(L_0M,L_0N)$ as the $i$-th homology of $L_0$ applied to the total complex of $F_\bullet \otimes_S G_\bullet$.
    As $S$ is $L_1$-acyclic and $F_\bullet \otimes_R G_\bullet$ consists levelwise of sums of shifts of $S$, it follows from \cref{lem:L1_sums} that its total complex is levelwise $L_1$-acyclic.
    % Finally, observe that the total complexes of $L_0(F_\bullet \otimes_S G_\bullet)$ and $L_0$ applied to the total complex of $F_\bullet\otimes_S G_\bullet$ 
    The claim therefore follows from \cref{lem:reversing_L0_homology}.
\end{proof}

\subsection{Filtered objects and spectral sequences in \texorpdfstring{$\Sp_{\K(1)}$}{Sp\_K(1)}}
\label{ssec:SS_in_K1_local_spectra}
As in, e.g., \cite[§1.2.2]{HA}, one can associate to a filtered object $X \colon \nerve(\Z_{\geq 0}) \to \Sp$ in spectra a spectral sequence with $E^1$-page
\[
    E^1_{s, t} = \pi_{s+t} \cofib \br*{X(s-1) \shortto X(s)},
\]
which converges to $\pi_{s+t}(\colim_n X(n))$.
In this section, we will consider the case where $X(n)$ is $\K(1)$-local for every $n$, and we will give an easy criterion in terms of this spectral sequence for when
\[
    L_{\K(1)} \colim_n X(n) = \colim_n X(n),
\]
which implies that the spectral sequence converges to $\pi_{s+t}(L_{\K(1)} \colim_n X(n))$.

\begin{lemma}
Let $X \colon \nerve(\Z_{\geq 0}) \to \Sp_{\K(1)}$ be a filtered object, and let $\set{E^r_{s,t}}$ be the associated spectral sequence.
Then for every $1 \leq r \leq \infty$, the $E^r$-page takes values in derived $p$-complete $\Z_p$-modules $\Modh_{\Z_p}$.
\end{lemma}
\begin{proof}
For $r=1$ this is clear, because the $E^1$-page is given by the homotopy groups of $\K(1)$-local spectra.
This then proves it for any finite $r$ by induction, as $\Modh_{\Z_p}$ is closed under finite limits and colimits in $\Mod_{\Z_p}$.
Finally, the $E^\infty$-page is given by the sequential colimit
\[
    E^\infty_{s,t} = \colim_r E^r_{s,t}.
\]
For $r$ large enough, the map $E^r_{s,t} \to E^{r+1}_{s,t}$ is an epimorphism.
The sequential colimit of a system of epimorphisms between derived $p$-complete modules is again derived $p$-complete.
Indeed, such a colimit is a quotient of one of the modules in the sequence, and a quotient of a derived $p$-complete module by a derived $p$-complete submodule is again derived $p$-complete.
This finishes the proof.
\end{proof}

\begin{proposition}
\label{prop:condition_convergence_SS_for_filtered_obj}
Let $X \colon \nerve(\Z_{\geq 0}) \to \Sp_{\K(1)}$ be a filtered object, and let $\set{E^r_{s,t}}$ be the associated spectral sequence.
If on the $E^\infty$-page every anti-diagonal {\upshape(}i.e., all $E^\infty_{s,t}$ with $s+t=n$ for fixed $n${\upshape)} contains only finitely many nonzero terms, then the spectral sequence converges to
\[
    \pi_{s+t}({\textstyle L_{\K(1)} \colim_n X(n)}).
\]
\end{proposition}
\begin{proof}
As in the proof of \cite[Prop.~1.2.2.14]{HA}, let $A_q := \pi_q \br{ \colim_n X(n) }$ and let
\[
    F_m A_q := \im \br*{ \pi_q X(m) \to \mathop{\pi_q} (\textstyle{ \colim_n X(n)}) }.
\]
Note that $F_mA_q$ is derived $p$-complete for every $m$ and $q$.
It is proved in loc.\ cit.\ that
\[
    E^\infty_{s,t} \cong F_sA_{s+t}/F_{s-1} A_{s+t}.
\]
By the assumption, we see that for fixed $q$ and for $m$ large enough, we have
\[
    F_mA_q = F_{m+1} A_q = \dotsb = A_q.
\]
This means that $\colim_m F_m A_q \cong \colim_m \pi_q X(m)$ is equal to $F_mA_q$ for some $m$, which in particular shows it is derived $p$-complete.
As a result, we have
\[
    L_0 \colim_m \pi_q X(m) \cong \colim_m \pi_q X(m) \cong \pi_q(\colim_m X(m)).
\]
By \cite[Cor.~3.5]{hovey_filtered_colimits_2008}, we have a short exact sequence
\[
     0 \to {\textstyle L_0 \colim_n \pi_q X(n)} \to \pi_q ({\textstyle L_{\K(1)} \colim_n X(n)}) \to {\textstyle L_1 \colim_n \pi_{q-1} X(n)} \to 0.
\]
The functor $L_1$ vanishes on derived $p$-complete modules by \cite[Thm.~A.6]{hovey_strickland_Ktheories_localisation}, so we find that $\colim_n X(n) = L_{\K(1)} \colim_n X(n)$, proving the claim.
\end{proof}

\begin{construction}
\label{rmk:Dold_Kan}
Suppose $Q_\bullet$ is a simplicial object in $\Sp_{\K(1)}$. Via the Dold--Kan correspondence of \cite[§1.2.4]{HA}, we obtain a filtered object
$X \colon \nerve(\Z_{\geq 0}) \to \Sp$ such that
\[
    X(m) \simeq \colim_{[k] \in \Delta_{\leq m}^{\op}} Q_k
\]
and $\colim_m X(m) \simeq \abs{Q_\bullet}$, where we take the realisation in $\Sp$.
Note that each $X(m)$ is $\K(1)$-local, as it is equivalent to a finite colimit of $\K(1)$-local spectra.
Hence it follows that if we consider the spectral sequence associated to $X$ and if the criterion from Proposition~\ref{prop:condition_convergence_SS_for_filtered_obj} applies, then the spectral sequence converges to $L_{\K(1)} \abs{Q_\bullet}$, i.e., the realisation in $\Sp_{\K(1)}$.
Moreover, by \cite[Rmk.~1.2.4.4]{HA}, the $E^1$-page of this spectral sequence can be described as follows: the column $E^{1}_{*, t}$ is isomorphic to the normalised chain complex associated to the simplicial abelian group $\pi_t(Q_\bullet)$.
\end{construction}

\subsection{The construction of the Tor spectral sequence}
\label{ssec:construction_Tor_SS}
In this section we construct the $\K(1)$-local Tor spectral sequence and prove Theorem~\ref{prop:K1_local_Tor_SS}.
We closely follow \cite[§7.2.1]{HA}.

Let $A$ be a $\K(1)$-local $\E_1$-ring spectrum, $M$ a $\K(1)$-local right $A$-module spectrum, and $N$ a $\K(1)$-local left $A$-module spectrum.
Consider the composite
\[
    \pi_* M \otimes_{\pi_* A} \pi_* N \to \pi_*(M \otimes_A N) \to \pi_*(M \hotimes_A N).
\]
As the target is derived $p$-complete, this factors through a map of $\Z$-graded derived $p$-complete $\Z_p$-modules
\[
    \pi_* M \hotimes_{\pi_* A} \pi_* N \to \pi_*(M \hotimes_A N).
\]

\begin{definition}
Let $F$ be a $\K(1)$-local left $A$-module.
We say $F$ is \defi{quasi-free} if $F \simeq L_{\K(1)} \bigoplus_{\alpha} \Sigma^{n_\alpha}A$ for some integers $n_\alpha$.
\end{definition}

\begin{lemma}
\label{lem:isomorphism_htpy_groups_smash_product}
Suppose $N$ is a quasi-free left $A$-module.
Then the map
\[
    \pi_* M \hotimes_{\pi_* A} \pi_* N \to \pi_*(M \hotimes_A N)
\]
is an isomorphism.
\end{lemma}
\begin{proof}
The functor $\pi_* \colon \Sp_{\K(1)} \to \Modh_{\Z_p}^*$ preserves small coproducts by \cite[Thm.~3.3]{hovey_filtered_colimits_2008}.
Hence both domain and target of the map are compatible with coproducts taken in the $\K(1)$-local category.
As both sides are also compatible with shifts, we are reduced to the case $N = A$, where it is obvious that the map is an isomorphism.
\end{proof}

\begin{proof}[Proof of Theorem~{\upshape\ref{prop:K1_local_Tor_SS}}]
Let $S = \set{\Sigma^n A \mid n \in \Z }$.
Choose an $S$-free $S$-hypercovering $P_\bullet$ of $N$ in the $\infty$-category $\mathop{L_{\K(1)}}\LMod_A$, which exists by \cite[Prop.~7.2.1.4]{HA}.
We consider the spectral sequence associated to the simplicial object $M \hotimes_A P_\bullet$ in $\Sp_{\K(1)}$ in the way described in Construction~\ref{rmk:Dold_Kan}.
The convergence criterion from Proposition~\ref{prop:condition_convergence_SS_for_filtered_obj} is satisfied by assumption, and therefore the spectral sequence converges to $\pi_* (M \hotimes_A N)$.

It remains to identify the $E^2$-page.
By Construction~\ref{rmk:Dold_Kan}, the column $E^1_{* ,t}$ is isomorphic to the normalised chain complex associated to the simplicial abelian group $\pi_t(M \hotimes_A P_\bullet)$.
The fact that $P_\bullet$ is $S$-free implies that $P_n$ is quasi-free for all $n \geq 0$.
Hence by Lemma~\ref{lem:isomorphism_htpy_groups_smash_product} we have an isomorphism
\[
    \pi_* M \hotimes_{\pi_* A} \pi_* P_\bullet \congto \pi_*(M \hotimes_A P_\bullet).
\]
Because $P_n$ is quasi-free, we also have that $\pi_* P_n$ is a projective object in $\LModh_{\pi_*A}^*$.
This combined with the fact that $P_\bullet$ is an $S$-hypercovering of $N$ implies that the normalised chain complex associated to $\pi_*P_\bullet$ is a resolution of $\pi_*N$ by projective objects of $\LModh_{\pi_*A}^*$.
Taking homology on the $E^1$-page, we therefore find
\[
    E^2_{s, t} \cong \Torh_{s, t}^{\pi_* A}(\pi_*M,\ \pi_*N).\qedhere
\]
\end{proof}